\documentclass[10pt,oneside,reqno,twoside]{amsart}
\usepackage[foot]{amsaddr} 
\usepackage[utf8]{inputenc}
\usepackage{amsmath}
\usepackage{amsfonts}
\usepackage{amssymb}
\usepackage{setspace}
\usepackage{amssymb}
\usepackage[shortlabels]{enumitem}
\usepackage{xcolor}
\usepackage{graphicx}
\usepackage{appendix}
\usepackage{amsthm}
\newtheorem{thm}{Theorem}[section]

\numberwithin{equation}{section}
\usepackage{bbm}
\newtheorem{rmk}{Remark}[section]

\newtheorem{prop}{Proposition}[section]
\newtheorem{lm}{Lemma}[section]
\newtheorem{cor}{Corollary}[section]
\usepackage{geometry}
\usepackage{fancyhdr}
\usepackage{hyperref}
\hypersetup{
    colorlinks=true,
    linkcolor=blue,
    filecolor=magenta,      
    urlcolor=blue,
    pdftitle={Overleaf Example},
    citecolor=red,
    pdfpagemode=FullScreen,
    }
\usepackage{mfirstuc}
\makeatletter
\def\@setauthors{%
  \begingroup
  \def\thanks{\protect\thanks@warning}%
  \trivlist
  \centering\footnotesize \@topsep30\p@\relax
  \advance\@topsep by -\baselineskip
  \item\relax
  \author@andify\authors
  \def\\{\protect\linebreak}%
  \authors%
  \ifx\@empty\contribs
  \else
    ,\penalty-3 \space \@setcontribs
    \@closetoccontribs
  \fi
  \endtrivlist
  \endgroup
}
\date{}
\begin{document}
\title[Coupled Backward Stochastic Parabolic Equations with One Control Force]{Null Controllability for Cascade systems of Coupled Backward Stochastic Parabolic Equations with One Distributed Control}

\author[S. Boulite]{\large Said Boulite$^1$}
\address{$^1$Cadi Ayyad University, National School of Applied Sciences, LMDP, UMMISCO (IRD-UPMC), Marrakesh, Morocco}
\email{s.boulite@uca.ma}

\author[A. Elgrou]{\large Abdellatif Elgrou$^2$}
\address{$^2$Cadi Ayyad University, Faculty of Sciences Semlalia, LMDP, UMMISCO (IRD-UPMC), Marrakesh, Morocco}
\email{abdoelgrou@gmail.com}
\email{maniar@uca.ma}

\author[L. Maniar]{\large Lahcen Maniar$^{2,3}$}
\address{$^3$University Mohammed VI  Polytechnic, Vanguard Center, Benguerir, Morocco}
\email{Lahcen.Maniar@um6p.ma}

\dedicatory{ \large Dedicated to the memory of Professor Hammadi Bouslous}

\begin{abstract}
We prove the null controllability of a cascade system of \(n\) coupled backward stochastic parabolic equations involving both reaction and convection terms, as well as general second-order parabolic operators, with \(n \geq 2\). To achieve this, we apply a single distributed control to the first equation, while the other equations are controlled through the coupling. To obtain our results, we develop a new global Carleman estimate for the forward stochastic parabolic adjoint system with some terms in the \(H^{-1}\)-space. Subsequently, we derive the appropriate observability inequality, and by employing the classical duality argument, we establish our null  controllability result. Additionally, we provide an estimate for the null control cost with respect to the final time \(T\) and the potentials.
\end{abstract}
\maketitle
\smallskip
\textbf{Keywords:} {Controllability, Stochastic parabolic systems, Carleman estimates, Observability.} 

\section{Introduction and main results}
Let $T>0$ and $(\Omega,\mathcal{F},\mathbf{F},\mathbb{P})$ be a fixed complete filtered probability space on which a one-dimensional standard Brownian motion $W(\cdot)$ is defined such that $\mathbf{F}=\{\mathcal{F}_t\}_{t\in[0,T]}$ is the natural filtration generated by $W(\cdot)$ and augmented by all the $\mathbb{P}$-null sets in $\mathcal{F}$. Let $\mathcal{X}$ be a Banach space; we denote by $L^2_{\mathcal{F}_t}(\Omega;\mathcal{X})$ the Banach space of all $\mathcal{X}$-valued $\mathcal{F}_t$-measurable random variables $X$ such that $\mathbb{E}\big(\vert X\vert_\mathcal{X}^2\big)<\infty$, with the canonical norm. $L^2_\mathcal{F}(0,T;\mathcal{X})$ denotes the Banach space of all $\mathcal{X}$-valued $\mathbf{F}$-adapted processes $X(\cdot)$ such that $\mathbb{E}\big(\vert X(\cdot)\vert^2_{L^2(0,T;\mathcal{X})}\big)<\infty$, with the canonical norm. $L^\infty_\mathcal{F}(0,T;\mathcal{X})$ represents the Banach space consisting of all $\mathcal{X}$-valued $\mathbf{F}$-adapted essentially bounded processes, with the canonical norm denoted by $|\cdot|_\infty$. The space $L^2_\mathcal{F}(\Omega;C([0,T];\mathcal{X}))$ indicates the Banach space of all $\mathcal{X}$-valued $\mathbf{F}$-adapted continuous processes $X(\cdot)$ such that $\mathbb{E}\big(\vert X(\cdot)\vert^2_{C([0,T];\mathcal{X})}\big)<\infty$, with the canonical norm and $C([0,T];\mathcal{X})$ denotes the Banach space of all $\mathcal{X}$-valued continuous functions defined on $[0,T]$. Similarly, one can define $L^\infty_\mathcal{F}(\Omega;C^m([0,T];\mathcal{X}))$ for any positive integer $m$. In the sequel, for every $\textbf{x},\textbf{y}\in\mathbb{R}^n$ ($n\geq1$), we denote by $|\textbf{x}|$ (resp., $(\textbf{x},\textbf{y})_{\mathbb{R}^n}$) the Euclidean norm of $\textbf{x}$ (resp., the Euclidean inner product of $\textbf{x}$ and $\textbf{y}$).

Let $G\subset\mathbb{R}^N$ ($N\geq1$) be an open and bounded domain with a $C^2$ boundary $\Gamma=\partial G$, and $G_0\Subset G$ be a given non-empty open subset strictly contained in $G$ (i.e., $\overline{G_0}\subset G$ where $\overline{G_0}$ denotes the closure of $G_0$). We indicate by $\chi_{G_0}$ the characteristic function of $G_0$ and $dx$ designates the Lebesgue measure in $G$. Set 
$$Q=(0,T)\times G, \,\,\,\quad \Sigma=(0,T)\times\Gamma\,\,\quad \textnormal{and}\,\,\,\quad Q_0=(0,T)\times G_0.$$ 

Let $n\geq1$ be an integer. The main goal of this paper is to study the null controllability of the following coupled backward stochastic parabolic system
\begin{equation}\label{1.1}
\begin{cases}
\begin{array}{ll}
dy_1+L_1(t)y_1\,dt=\left[\displaystyle\sum_{j=1}^n a_{1j}y_j+\sum_{j=1}^n C_{1j}\cdot\nabla y_j+\sum_{j=1}^n b_{1j}Y_j+\chi_{G_0}(x)u\right]\,dt+Y_1\,dW(t)&\textnormal{in}\,\,Q,\\
dy_2+L_2(t)y_2\,dt=\left[\displaystyle\sum_{j=1}^n a_{2j}y_j+\sum_{j=1}^n C_{2j}\cdot\nabla y_j+\sum_{j=1}^n b_{2j}Y_j\right]\,dt+Y_2\,dW(t)&\textnormal{in}\,\,Q,\\
\hspace{0.7cm}\vdots\\
dy_n+L_n(t)y_n\,dt=\left[\displaystyle\sum_{j=1}^n a_{nj}y_j+\sum_{j=1}^n C_{nj}\cdot\nabla y_j+\sum_{j=1}^n b_{nj}Y_j\right]\,dt+Y_n\,dW(t)&\textnormal{in}\,\,Q,\\
y_i=0,\quad 1\leq i\leq n&\textnormal{on}\,\,\Sigma,\\
y_i(T)=y_{T,i},\quad 1\leq i\leq n&\textnormal{in}\,\,G,
			\end{array}
		\end{cases}
\end{equation}
where $y_{T,i}\in L^2_{\mathcal{F}_T}(\Omega;L^2(G))$ ($1\leq i\leq n$) are the finale states, $u\in L^2_\mathcal{F}(0,T;L^2(G_0))$ is the control variable and the coefficients $a_{ij}$, $C_{ij}$ and $b_{ij}$ are assumed to be
\begin{align*}
&a_{ij}=a_{ij}(\omega,t,x)\in L_\mathcal{F}^\infty(0, T; L^\infty(G)), \quad C_{ij}=C_{ij}(\omega,t,x)\in L_\mathcal{F}^\infty(0, T;L^\infty(G;\mathbb{R}^N)),\\
&\hspace{2cm}b_{ij}=b_{ij}(\omega,t,x)\in L_\mathcal{F}^\infty(0, T; L^\infty(G)),\quad\quad 1\leq i,j\leq n.
\end{align*}
It is easy to see that system \eqref{1.1} can be written as
\begin{equation}\label{1.1Abst}
\begin{cases}
\begin{array}{ll}
dy+L(t)y\,dt=\left(Ay+C\cdot\nabla y+BY+D\chi_{G_0}(x)u\right)\,dt+Y\,dW(t)&\textnormal{in}\,\,Q,\\
y=0&\textnormal{on}\,\,\Sigma,\\
y(T)=y^T&\textnormal{in}\,\,G,
			\end{array}
		\end{cases}
\end{equation}
where $(y,Y)=((y_i)_{1\leq i\leq n},(Y_i)_{1\leq i\leq n})$ is the state variable, $y^T=(y_{T,i})_{1\leq i\leq n}\in L^2_{\mathcal{F}_T}(\Omega;L^2(G;\mathbb{R}^n))$ is the final state, $\nabla y=(\nabla y_i)_{1\leq i\leq n}$, and the coefficients $A,C,B$ and $D$ are defined as follows
$$A=(a_{ij})_{1\leq i,j\leq n}\in L_\mathcal{F}^\infty(0, T; L^\infty(G;\mathbb{R}^{n\times n})),\qquad C=(C_{ij})_{1\leq i,j\leq n}\in L_\mathcal{F}^\infty(0, T; L^\infty(G;\mathbb{R}^{nN\times n})),$$ 
$$B=(b_{ij})_{1\leq i,j\leq n}\in L_\mathcal{F}^\infty(0, T; L^\infty(G;\mathbb{R}^{n\times n})),\qquad D\equiv e_1=(1,0,...,0)^*\in\mathbb{R}^n,$$ 
and $L(t)=\textnormal{diag}(L_1(t),...,L_n(t))$ is the matrix operator with $L_k$ are the second order parabolic operators defined by
\begin{align}\label{opLk}
L_k(t)y=\displaystyle\sum_{i,j=1}^N \frac{\partial}{\partial x_i}\left(\beta_{ij}^k(t,x)\frac{\partial y}{\partial x_j}\right), \qquad 1\leq k\leq n,
\end{align}
where $\beta_{ij}^k:\Omega\times Q\rightarrow\mathbb{R}$ satisfy the following assumptions:
\begin{enumerate}
\item $\beta_{ij}^k\in L^\infty_\mathcal{F}(\Omega;C^1([0,T];W^{1,\infty}(G)))$ and $\beta_{ij}^k=\beta_{ji}^k$, for any $1\leq i,j\leq N$.
\item  There exists a constant $\beta_0>0$ such that 
\begin{align}\label{assmponalpha}
\sum_{i,j=1}^N \beta_{ij}^k\xi_i\xi_j\geq \beta_0|\xi|^2\qquad\textnormal{for any}\quad (\omega,t,x,\xi)\in \Omega\times Q\times\mathbb{R}^N,
\end{align}
with $x=(x_1,...,x_N)$ and $\xi=(\xi_1,...,\xi_N)$.
\end{enumerate}

In this paper, we study the case when the matrices $A$, $C$ and $B$ have the following form
\begin{align}\label{assmAC}
\begin{aligned}
&A=\begin{pmatrix}
a_{11} & a_{12} & a_{13} & \cdots & a_{1n} \\
a_{21} & a_{22} & a_{23} & \cdots & a_{2n} \\
0 & a_{32} & a_{33} & \cdots & a_{3n} \\
\vdots & \ddots & \ddots & \ddots & \vdots \\
0 & 0 & \cdots & a_{n,n-1} & a_{nn}
\end{pmatrix},\qquad
C=\begin{pmatrix}
C_{11} & C_{12} & C_{13} & \cdots & C_{1n} \\
0 & C_{22} & C_{23} & \cdots & C_{2n} \\
0 & 0 & C_{33} & \cdots & C_{3n} \\
\vdots & \vdots & \vdots & \ddots & \vdots \\
0 & 0 & 0 & \cdots & C_{nn}
\end{pmatrix},\\
&\hspace{4cm}B=\begin{pmatrix}
b_{11} & b_{12} & b_{13} & \cdots & b_{1n} \\
0 & b_{22} & b_{23} & \cdots & b_{2n} \\
0 & 0 & b_{33} & \cdots & b_{3n} \\
\vdots & \vdots & \vdots & \ddots & \vdots \\
0 & 0 & 0 & \cdots & b_{nn}
\end{pmatrix},
\end{aligned}
\end{align}
where 
\begin{align*}
a_{ij}\in L_\mathcal{F}^\infty(0, T; L^\infty(G)),\quad C_{ij}\in L_\mathcal{F}^\infty(0, T; L^\infty(G;\mathbb{R}^N)), \\b_{ij}\in L_\mathcal{F}^\infty(0, T; L^\infty(G))
\quad\textnormal{for} \quad1\leq i\leq j\leq n,
\end{align*}
and the coupling terms
\begin{align*}
a_{i,i-1}\in L_\mathcal{F}^\infty(0, T; L^\infty(G)) \quad\textnormal{for}\quad 2\leq i\leq n
\end{align*}
satisfy the following assumption: There exists a constant $a_0>0$ such that
\begin{align}\label{asspmt}
    a_{i,i-1}\geq a_0 \quad\textnormal{or}\quad -a_{i,i-1}\geq a_0,\quad\textnormal{in}\quad(0,T)\times\widetilde{G}_0,\;\;\mathbb{P}\textnormal{-a.s.},
\end{align}
for any open set $\widetilde{G}_0\subset G_0$. In what follows, we denote by $M_0=\max_{2\leq i\leq n}|a_{i,i-1}|_\infty$.\\

Throughout this paper, $C_0$ denotes a positive constant depending only on $G$, $\widetilde{G}_0$, $a_0$, $\beta_0$, $M_0$ and $\beta_{ij}^k$, which may change from one place to another.

From \cite{zhou1992}, we have that \eqref{1.1} is well-posed i.e., for any terminal state
$y^T\in L^2_{\mathcal{F}_T}(\Omega;L^2(G;\mathbb{R}^n))$ and  control $u\in L^2_\mathcal{F}(0,T;L^2(G_0))$, the system \eqref{1.1} admits a unique weak solution 
$$(y,Y)\in \left(L^2_\mathcal{F}(\Omega;C([0,T];L^2(G;\mathbb{R}^n)))\bigcap L^2_\mathcal{F}(0,T;H^1_0(G;\mathbb{R}^n))\right)\times L^2_\mathcal{F}(0,T;L^2(G;\mathbb{R}^n)).$$
Moreover, there exists a constant $\mathcal{C}>0$ so that
\begin{align*}
&\,|y|_{L^2_\mathcal{F}(\Omega;C([0,T];L^2(G;\mathbb{R}^n)))}+|y|_{L^2_\mathcal{F}(0,T;H^1_0(G;\mathbb{R}^n))}+|Y|_{L^2_\mathcal{F}(0,T;L^2(G;\mathbb{R}^n))}\\
&\leq \mathcal{C}\left(|y^T|_{L^2_{\mathcal{F}_T}(\Omega;L^2(G;\mathbb{R}^n))}+|u|_{L^2_\mathcal{F}(0,T;L^2(G_0))}\right).
\end{align*}

The main result of this paper is the following null controllability of \eqref{1.1}.
\begin{thm}\label{thm01.1}
Let us assume that \eqref{asspmt} holds. Then, for any terminal state $y^T\in L^2_{\mathcal{F}_T}(\Omega;L^2(G;\mathbb{R}^n))$, there exist a control $\widehat{u}\in L^2_\mathcal{F}(0,T;L^2(G_0))$ such that the corresponding solution $(\widehat{y},\widehat{Y})$ of system \eqref{1.1} satisfies that $$\widehat{y}(0,\cdot) = 0\,\,\,\; \textnormal{in}\,\,G,\quad\mathbb{P}\textnormal{-a.s.}$$
Moreover, the control $\widehat{u}$ can be chosen so that
\begin{align}\label{1.2201}
\begin{aligned}
\vert \widehat{u}\vert_{L^2_\mathcal{F}(0,T;L^2(G))}\leq \sqrt{\exp(C_0K)}\,\left\vert y^T\right\vert_{L^2_{\mathcal{F}_T}(\Omega;L^2(G;\mathbb{R}^n))},
\end{aligned}
\end{align}
where the constant $K$ has the following form
\begin{align*}
K=1+T+\frac{1}{T}+\max_{i\leq j}\left(|a_{ij}|_\infty^{\frac{2}{3(j-i)+3}}+|C_{ij}|_\infty^{\frac{2}{3(j-i)+1}}+|b_{ij}|_\infty^{\frac{2}{3(j-i)+1}}+T\left(|a_{ij}|_\infty+|C_{ij}|^2_\infty+|b_{ij}|_\infty^2\right)\right).
\end{align*}
\end{thm} 
\begin{rmk}
The null controllability of some particular cases of system \eqref{1.1} has been studied in the following works:
    \begin{enumerate}[1.]
        \item In \cite{LiQi}, with
        \begin{align*}
            n = 2, \qquad C = 0, \qquad \beta_{ij}^k(t,x) = \delta_{ij} = 
            \begin{cases}
                1, & i = j, \\
                0, & i \neq j.
            \end{cases}, \qquad B = \textnormal{diag}(b_1, b_2).
        \end{align*}
 \item In \cite{Fadilicasc}, with
        \begin{align*}
            n \geq 2, \qquad C = 0, \qquad \beta_{ij}^k(t,x) = \delta_{ij}, \qquad B = \textnormal{diag}(b_1, b_2, \ldots, b_n).
        \end{align*}
    \end{enumerate}
Unlike in \cite{Fadilicasc,LiQi}, this paper addresses the controllability of more general backward stochastic coupled reaction-convection-diffusion equations of the form \eqref{1.1}, under the assumptions \eqref{opLk} and \eqref{assmAC}. Furthermore, to the authors' knowledge, this is the first work that provides an estimate (the estimate \eqref{1.2201}) for the cost of the null controllability of system \eqref{1.1} with explicit dependence on the final time \(T\) and the potentials \(a_{ij}\), \(b_{ij}\), and \(C_{ij}\) (with \(1 \leq i\leq j \leq n\)). Therefore, it is important to study the optimality of the constant \(\exp(C_0 K)\) in \eqref{1.2201}.
\end{rmk}

The null controllability property has been extensively studied for deterministic parabolic equations, with numerous results in the literature addressing this issue. For instance, we refer to \cite{ammBDG, coron07, fursikov1996controllability, GonBurTer, imanuvilov2003carleman, lions1972some, Zab95} for key findings in this area. In the context of deterministic coupled parabolic systems, the survey article \cite{surveyAmmarKBGT} provides an overview of various findings related to the controllability of such equations. Additionally, in \cite{GonBurTer}, the authors establish null controllability for a general class of cascade parabolic systems using a single control force. For further details and additional results concerning the controllability of fully coupled systems, we refer to \cite{ammBDG,GuerreroS}.

In the stochastic setting, some results have been established regarding the controllability of stochastic systems; we refer to the book \cite{lu2021mathematical} and the references therein. The case of forward and backward stochastic parabolic equations is also widely studied. In \cite{barbu2003carleman}, the authors presented some controllability results for a class of stochastic heat equations. Furthermore, \cite{tang2009null} was the first to establish null controllability for general forward and backward stochastic parabolic equations with Dirichlet boundary conditions, deriving a stochastic version of Carleman estimates through a weighted identity method. For an improved Carleman inequality for forward stochastic parabolic equations, we refer to \cite{liu2014global}, which employs a duality method and gives controllability for backward stochastic parabolic equations with bounded zero-order potentials. Additionally, \cite{convectermsEBBm} discusses controllability results for more general stochastic parabolic equations, incorporating zero and first-order terms. Controllability for the case of Robin boundary conditions has also been explored in \cite{RobinBCelg, yan2018carleman}. For the case of dynamic boundary conditions, referred to as Wentzell boundary conditions, see \cite{SPEwithDBC, BackSPEwithDBC}. The controllability of general forward stochastic parabolic equations remains an open problem; for more details on this issue, see \cite{convectermsEBBm, lu2021mathematical, tang2009null}. In the existing literature, only partial results have been established concerning the controllability of forward stochastic parabolic equations when the potentials are space-independent; for that, see e.g., \cite{lu2011some, observineqback}.

For coupled stochastic parabolic equations, there are limited results concerning the controllability of forward and backward stochastic parabolic systems. In \cite{LiQi}, the authors investigated the null controllability of a system of two coupled backward stochastic parabolic equations with one control. Subsequently, \cite{Fadilicasc} generalized this work and established the controllability of a cascade system of \(n\) coupled backward stochastic parabolic equations with \(n \geq 2\). In \cite{Fadilicasc, LiQi}, the equations incorporate only reaction terms and Laplacian operators.  There are also other controllability results for some coupled stochastic parabolic equations. We refer to \cite{wangNull24} for coupled fourth-order backward stochastic parabolic equations, \cite{LiuuLiuX} for a forward stochastic parabolic system, and \cite{liu14couplfor} for coupled fractional 
 stochastic parabolic equations.  Additionally, \cite{liu2014global, yan20111} studied insensitizing control problems as a controllability issue for a suitable cascade system of coupled forward and backward stochastic parabolic equations.

To the best of our knowledge, the present paper is the first to address the controllability of a more general framework of cascade systems of coupled backward stochastic parabolic equations that includes general second-order parabolic operators as well as both reaction and convection terms. Moreover, we establish the estimate \eqref{1.2201} for the null control cost of system \eqref{1.1}.\\

To prove the null controllability of \eqref{1.1}, we introduce the following adjoint equation
\begin{equation}\label{1.1Abstractadjoint}
\begin{cases}
\begin{array}{ll}
dz-L(t)z\,dt=\left(-A^*z+\nabla\cdot(C^*z)\right)\,dt-B^*z \,dW(t)&\textnormal{in}\,\,Q,\\
z=0&\textnormal{on}\,\,\Sigma,\\
z(0)=z^0&\textnormal{in}\,\,G,
			\end{array}
		\end{cases}
\end{equation}
where $z=(z_1,...,z_n)$ is the state variable, $z^0\in L^2_{\mathcal{F}_0}(\Omega;L^2(G;\mathbb{R}^n))$ is the initial state and $A^*$ (resp., $B^*$ and $C^*$) is the transpose of $A$ (resp., $B$ and $C$). 

Under the assumption \eqref{assmAC}, it is easy to see that the above adjoint system \eqref{1.1Abstractadjoint} becomes
\begin{equation}\label{1.1Adjoint}
\begin{cases}
\begin{array}{ll}
dz_i-L_i(t)z_i\,dt=\left[-\displaystyle\sum_{j=1}^i \left(a_{ji}z_j-\nabla\cdot(C_{ji}z_j)\right)-a_{i+1,i}z_{i+1}\right]\,dt-\displaystyle\sum_{j=1}^i b_{ji}z_j\,dW(t)&\textnormal{in}\,\,Q,\\
...\hspace{9.4cm}(1\leq i\leq n-1),\\
dz_n-L_n(t)z_n\,dt=\left[-\displaystyle\sum_{j=1}^n \left(a_{jn}z_j-\nabla\cdot(C_{jn}z_j)\right)\right]\,dt-\displaystyle\sum_{j=1}^n b_{jn}z_j \,dW(t)&\textnormal{in}\,\,Q,\\
z_i=0,\quad 1\leq i\leq n&\textnormal{on}\,\,\Sigma,\\
z_i(0)=z_{0,i},\quad 1\leq i\leq n&\textnormal{in}\,\,G.
			\end{array}
		\end{cases}
\end{equation}
From \cite{krylov1994}, for any $z^0=(z_{0,i})_{1\leq i\leq n}\in L^2_{\mathcal{F}_0}(\Omega;L^2(G;\mathbb{R}^n))$, the system \eqref{1.1Adjoint} has a unique weak solution
$$z=(z_i)_{1\leq i\leq n}\in L^2_\mathcal{F}(\Omega;C([0,T];L^2(G;\mathbb{R}^n)))\bigcap L^2_\mathcal{F}(0,T;H^1_0(G;\mathbb{R}^n)).$$
Moreover, there exists a positive constant $\mathcal{C}$ so that
\begin{align*}
&\,|z|_{L^2_\mathcal{F}(\Omega;C([0,T];L^2(G;\mathbb{R}^n)))}+|z|_{L^2_\mathcal{F}(0,T;H^1_0(G;\mathbb{R}^n))}\leq \mathcal{C}|z^0|_{L^2_{\mathcal{F}_0}(\Omega;L^2(G;\mathbb{R}^n))}.
\end{align*}

By the classical duality argument, we know that the null controllability of \eqref{1.1} (with adapted $L^2$-controls) is equivalent to the following observability inequality of \eqref{1.1Adjoint}.
\begin{prop}\label{propo1.1}
There exists a constant $C_0>0$ such that for every $z^0\in L^2_{\mathcal{F}_0}(\Omega;L^2(G;\mathbb{R}^n))$, the associated solution $z=(z_1,...,z_n)$ of system \eqref{1.1Adjoint} satisfies that
\begin{align}\label{observineqqke}
\mathbb{E}\int_{G} |z(T,x)|^2 dx\leq \exp(C_0K) \,\mathbb{E}\int_{Q_0} |z_1|^2 \,dx\,dt,
\end{align}
where $C_0$ and $K$ are the same constants as in \eqref{1.2201}.
\end{prop}

To prove the observability inequality \eqref{observineqqke}, the key tool is to establish the following new global Carleman estimate for system \eqref{1.1Adjoint}.
\begin{thm}\label{thmm1.2}
Let us assume that \eqref{asspmt}. Then, there exists a large $\mu_0>0$ so that for $\mu=\mu_0$, one can find a constant $C_0>0$ (depending on $G$, $\widetilde{G}_0$, $\mu_0$, $a_0$, $\beta_0$, $M_0$ and $\beta_{ij}^k$) and $l\geq3$ (which depends only on $n$) such that, for any initial state $z^0\in L^2_{\mathcal{F}_0}(\Omega;L^2(G;\mathbb{R}^n))$, the associated solution $z=(z_1,...,z_n)$ of \eqref{1.1Adjoint} satisfies that
    \begin{align}\label{carestimate}
        \sum_{i=1}^n \mathcal{I}(3(n+1-i),z_i)\leq C_0 \lambda^l\mathbb{E}\int_0^T\int_{\widetilde{G}_0} \theta^2\gamma^l|z_1|^2 \,dx\,dt,
    \end{align}
    for any     $$\lambda\geq\lambda_0=C_0\left(T+T^2+T^2\max_{i\leq j}\left(|a_{ij}|^{\frac{2}{3(j-i)+3}}_\infty+|C_{ij}|^{\frac{2}{3(j-i)+1}}_\infty+|b_{ij}|^{\frac{2}{3(j-i)+1}}_\infty\right)\right).$$ 
    With
    $$\mathcal{I}(d,z_i)=\lambda^d\mathbb{E}\int_Q \theta^2\gamma^d|z_i|^2\,dx\,dt + \lambda^{d-2}\mathbb{E}\int_Q \theta^2\gamma^{d-2}\vert\nabla z_i\vert^2\,dx\,dt,$$
    where $\theta$ and $\gamma$ are the weight functions defined in \eqref{2.2012}.
\end{thm}
From the above Carleman estimate \eqref{carestimate}, one can easily deduce the following unique continuation property for solutions of \eqref{1.1Adjoint}.
\begin{cor}\label{corr1}
Suppose that the assumption \eqref{asspmt} holds. Then any solution \( z = (z_1, \ldots, z_n) \) of system \eqref{1.1Adjoint} fulfills  that
\[
z_1 = 0 \;\;\textnormal{in}\;\; Q, \;\;\mathbb{P}\textnormal{-a.s.,} \quad \Longrightarrow \quad z = 0 \;\;\textnormal{in}\;\; Q, \;\;\mathbb{P}\textnormal{-a.s.}
\]
\end{cor}
As a direct consequence of Corollary \ref{corr1} and by employing a duality argument, we also obtain the following approximate controllability result for system \eqref{1.1}.
\begin{prop}
Assume that \eqref{asspmt} holds. Then, for any terminal state \( y^T \in L^2_{\mathcal{F}_T}(\Omega;L^2(G;\mathbb{R}^n)) \), any initial state \( y^0 \in L^2_{\mathcal{F}_0}(\Omega;L^2(G;\mathbb{R}^n)) \), and all \( \varepsilon > 0 \), there exists a control \( \widehat{u} \in L^2_\mathcal{F}(0,T;L^2(G_0)) \) such that the corresponding solution \( (\widehat{y}, \widehat{Y}) \) of \eqref{1.1} satisfies that
\[
\mathbb{E}\left|\widehat{y}(0) - y^0\right|^2_{L^2(G;\mathbb{R}^n)} \leq \varepsilon.
\]
\end{prop}

Theorem \ref{thmm1.2} will be the key tool for establishing the observability inequality \eqref{observineqqke} and, subsequently, our null controllability result in Theorem \ref{thm01.1}. To prove Theorem \ref{thmm1.2}, we will first derive a general global Carleman inequality for solutions of forward stochastic parabolic equations with a general parameter $d\in\mathbb{R}$ and a drift term in the negative Sobolev space, specifically in \( L^2_\mathcal{F}(0,T,H^{-1}(G)) \). Subsequently, we prove an intermediate result in Lemma \ref{lmm2.3}, where we establish an estimate for each component \( z_k \) (\(2 \leq k \leq n\)) of the solution \( z \) of system \eqref{1.1Adjoint} with respect to the previous components \( z_i \) (\(1 \leq i \leq k-1\)). Therefore, by combining this intermediate result and our previous general Carleman estimate, we will finally prove the Carleman estimate \eqref{carestimate}.\\

Now, some remarks are in order.
\begin{rmk}
The assumption \eqref{asspmt} is a sufficient condition for the null controllability result in Theorem \ref{thm01.1}. Therefore, it is essential to investigate when such a condition is also necessary. For details regarding some coupled deterministic parabolic equations, see \cite{surveyAmmarKBGT,GonBurTer}, where it is established that, in the case of constant coefficients, the assumption \eqref{asspmt} is indeed necessary.
\end{rmk}
\begin{rmk}
The assumption \eqref{asspmt} seems to be a very strong condition for controllability. Therefore, it is important to study the case of some weaker conditions of \eqref{asspmt}, such as one of the following assumptions:
\begin{enumerate}[a)]
\item For any $i=2,3,...,n$
\[
a_{i,i-1} \neq 0, \quad \textnormal{in} \quad \widetilde{G}_0 \times (0,T), \;\;\mathbb{P}\textnormal{-a.s.} 
\]
\item There exists a constant $a_0>0$ such that for any $i=2,3,...,n$
\[
|a_{i,i-1}| \geq a_0, \quad \textnormal{in} \quad \widetilde{G}_0 \times (0,T), \;\;\mathbb{P}\textnormal{-a.s.}
\]
\end{enumerate}
For some known results in the deterministic case, using a different technique called the fictitious control method, we refer to \cite{dupLissy}.
\end{rmk}
\begin{rmk}
We have extensively used the form of the cascade system of the adjoint equation \eqref{1.1Adjoint} to derive the observability inequality \eqref{observineqqke}. The problem of the null controllability of \eqref{1.1} is widely open in the case of general coupling matrices \( A \), \( B \), and \( C \) and general control vector $D\in\mathbb{R}^n$. Moreover, it is quite interesting to study the controllability of some coupled systems such as \eqref{1.1} when the control \( u \) is acting on the boundary.
\end{rmk}

The rest of this paper is organized as follows. In the next section, we establish the Carleman estimate \eqref{carestimate}, i.e., Theorem \ref{thmm1.2}. In Section \ref{sec3}, we prove the observability inequality \eqref{observineqqke} and deduce our null controllability result for system \eqref{1.1}, i.e., Theorem \ref{thm01.1}.

\section{Carleman estimates: Proof of Theorem \ref{thmm1.2}}\label{sec2}
This section is devoted to proving the global Carleman estimate \eqref{carestimate}. We first provide the following known lemma. For the proof, see e.g., \cite{coron07,fursikov1996controllability}.
\begin{lm}\label{lm2.1}
For any nonempty open subset $G_1\Subset G$, there exists a function $\psi\in C^4(\overline{G})$ such that
\begin{equation}
 \psi(x)>0 \,\,\;\textnormal{in}\,\,G,\quad\quad\psi(x)=0\,\;\,\textnormal{on}\,\,\Gamma,\quad\quad \vert\nabla\psi(x)\vert>0\,\,\;\textnormal{in}\,\,\overline{G\setminus G_1}.
\end{equation}
\end{lm}
For any parameters $\lambda, \mu\geq1$, we define the following weight functions
\begin{equation}\label{2.2012}
\alpha\equiv\alpha(t,x) = \frac{e^{\mu\psi(x)}-e^{2\mu\vert\psi\vert_\infty}}{t(T-t)},\quad\; \gamma\equiv\gamma(t)=\frac{1}{t(T-t)},\quad\; \theta\equiv\theta(t,x)=e^{\lambda\alpha}.
\end{equation}
It is easy to check that there exists a constant $C_0=C_0(G)>0$ so that for all $(t,x)\in Q$ 
\begin{equation}\label{2.301}
\begin{array}{l}
\gamma(t)\geq \frac{4}{T^2},\qquad\vert\gamma_t(t)\vert\leq C_0T\gamma^2(t),\qquad\vert\gamma_{tt}(t)\vert\leq C_0T^2\gamma^3(t),\\\\
\vert\alpha_t(t,x)\vert\leq C_0Te^{2\mu\vert\psi\vert_\infty}\gamma^2(t),\qquad\vert\alpha_{tt}(t,x)\vert\leq C_0T^2e^{2\mu\vert\psi\vert_\infty}\gamma^3(t).
		\end{array}
	\end{equation}
\subsection{General Carleman estimate for forward stochastic parabolic equations}
Let us consider the following forward stochastic parabolic equation
\begin{equation}\label{3.010}
		\begin{cases}
			\begin{array}{ll}
			dz- L_0(t) z \,dt = (F_0+\nabla\cdot F)\,dt + F_1 \,dW(t)	 &\textnormal{in}\,\,Q,\\
				z=0 &\textnormal{on}\,\,\Sigma,\\
			 z(0)=z_0 &\textnormal{in}\,\,G,
			\end{array}
		\end{cases}
	\end{equation}
where $z_0\in L^2_{\mathcal{F}_0}(\Omega;L^2(G))$, $F_0,F_1\in L^2_\mathcal{F}(0,T;L^2(G))$, $F\in L^2_\mathcal{F}(0,T;L^2(G;\mathbb{R}^N))$, and the operator $L_0$ is given by
$$L_0(t)y=\displaystyle\sum_{i,j=1}^N \frac{\partial}{\partial x_i}\left(\beta_{ij}^0(t,x)\frac{\partial y}{\partial x_j}\right),$$
where $\beta_{ij}^0\in L^\infty_\mathcal{F}(\Omega;C^1([0,T];W^{1,\infty}(G)))$, $\beta_{ij}^0=\beta_{ji}^0$ ($1\leq i,j\leq N$), and there exists a positive constant $\beta^0$ so that 
$$\sum_{i,j=1}^N \beta_{ij}^0\xi_i\xi_j\geq \beta^0|\xi|^2\qquad\textnormal{for any}\quad (\omega,t,x,\xi)\in \Omega\times Q\times\mathbb{R}^N.$$

We have the following Carleman estimate for equation \eqref{3.010}. For the proof, see \cite[Theorem 3.3]{convectermsEBBm}.
\begin{thm}\label{thm1.1}
Let $\mathcal{O}\subset G$ be a nonempty open subset. Then, there exist a large $\mu_0\geq1$ so that for $\mu=\mu_0$, one can find a positive constant $C_1=C_1(G,\mathcal{O},\mu_0,\beta^0,\beta_{ij}^0)$ such that for all  $F_0,F_1\in L^2_\mathcal{F}(0,T;L^2(G))$, $F\in L^2_\mathcal{F}(0,T;L^2(G;\mathbb{R}^N))$ and $z_0\in L^2_{\mathcal{F}_0}(\Omega;L^2(G))$, the corresponding weak solution $z$ of \eqref{3.010} satisfies that for all $\lambda\geq C_1(T+T^2)$
\begin{align}
\begin{aligned}
&\,\lambda^3\mathbb{E}\int_Q \theta^2\gamma^3z^2\,dx\,dt + \lambda\mathbb{E}\int_Q \theta^2\gamma\vert\nabla z\vert^2\,dx\,dt\\
&\leq C_1\Bigg[\lambda^3\mathbb{E}\int_{0}^T\int_\mathcal{O} \theta^2\gamma^3z^2\,dx\,dt+\mathbb{E}\int_Q \theta^2F_0^2\,dx\,dt\\
\label{3.202002}&\quad\quad\,\;+\lambda^2\mathbb{E}\int_Q \theta^2\gamma^2F_1^2\,dx\,dt+\lambda^2\mathbb{E}\int_Q \theta^2\gamma^2\vert F\vert^2\,dx\,dt\Bigg].
\end{aligned}
\end{align}
\end{thm}
In what follows, we fix $\mu=\mu_0$ given in Theorem \ref{thm1.1}. The following estimate provides a generalization of the previous Carleman estimate \eqref{3.202002}.
\begin{lm}\label{lm1.1}
Let $\mathcal{O}\subset G$ be a nonempty open subset and $d\in\mathbb{R}$. Then, one can find a positive constant $C_2=C_2(G,\mathcal{O},\mu_0,\beta^0,\beta_{ij}^0,d)$ such that for all  $F_0,F_1\in L^2_\mathcal{F}(0,T;L^2(G))$, $F\in L^2_\mathcal{F}(0,T;L^2(G;\mathbb{R}^N))$ and $z_0\in L^2_{\mathcal{F}_0}(\Omega;L^2(G))$, the associated weak solution $z$ of equation \eqref{3.010} satisfies that
\begin{align}\label{3.202002gene}
\begin{aligned}
&\,\lambda^d\mathbb{E}\int_Q \theta^2\gamma^dz^2\,dx\,dt + \lambda^{d-2}\mathbb{E}\int_Q \theta^2\gamma^{d-2}\vert\nabla z\vert^2\,dx\,dt\\
&\leq C_2\Bigg[\lambda^d\mathbb{E}\int_{0}^T\int_\mathcal{O} \theta^2\gamma^dz^2\,dx\,dt+\lambda^{d-3}\mathbb{E}\int_Q \theta^2\gamma^{d-3}F_0^2\,dx\,dt\\
&\quad\quad\,\;+\lambda^{d-1}\mathbb{E}\int_Q \theta^2\gamma^{d-1}F_1^2\,dx\,dt+\lambda^{d-1}\mathbb{E}\int_Q \theta^2\gamma^{d-1}\vert F\vert^2\,dx\,dt\Bigg],
\end{aligned}
\end{align}
for all $\lambda\geq C_2(T+T^2)$.
\end{lm}
\begin{proof}
Notice that if $d=3$, the inequality \eqref{3.202002gene} is shown in Theorem \ref{thm1.1}. For $d\neq3$, let us consider the function $h=(\lambda\gamma)^{\frac{d-3}{2}}z$. It is easy to see that
\begin{equation}\label{eqqwih}
		\begin{cases}
			\begin{array}{ll}
			dh- L_0(t) h \,dt = \left[(\lambda\gamma)^{\frac{d-3}{2}}F_0+\frac{d-3}{2}\gamma'\gamma^{-1}h+(\lambda\gamma)^{\frac{d-3}{2}}\nabla\cdot F\right]\,dt + (\lambda\gamma)^{\frac{d-3}{2}}F_1 \,dW(t)	 &\textnormal{in}\,\,Q,\\
				h=0 &\textnormal{on}\,\,\Sigma.
			\end{array}
		\end{cases}
	\end{equation}
We now apply Carleman estimate \eqref{3.202002} for the equation \eqref{eqqwih}, we obtain that
\begin{align}\label{inn1}
\begin{aligned}
&\,\lambda^3\mathbb{E}\int_Q \theta^2\gamma^3h^2\,dx\,dt + \lambda\mathbb{E}\int_Q \theta^2\gamma\vert\nabla h\vert^2\,dx\,dt\\
&\leq C_1\Bigg[\lambda^3\mathbb{E}\int_{0}^T\int_\mathcal{O} \theta^2\gamma^3h^2\,dx\,dt+\mathbb{E}\int_Q \theta^2\left|(\lambda\gamma)^{\frac{d-3}{2}}F_0+\frac{d-3}{2}\gamma'\gamma^{-1}h\right|^2\,dx\,dt\\
&\quad\quad\,\;+\lambda^2\mathbb{E}\int_Q \theta^2\gamma^2\left|(\lambda\gamma)^{\frac{d-3}{2}}F_1\right|^2\,dx\,dt+\lambda^2\mathbb{E}\int_Q \theta^2\gamma^2\left|(\lambda\gamma)^{\frac{d-3}{2}} F\right|^2\,dx\,dt\Bigg],
\end{aligned}
\end{align}
for any $\lambda\geq C_1(T+T^2)$. Using that $|\gamma'\gamma^{-1}|\leq C_0T\gamma$ for the second term on the right hand side of \eqref{inn1}, we get that
\begin{align*}
\begin{aligned}
&\,\mathbb{E}\int_Q \theta^2\left|(\lambda\gamma)^{\frac{d-3}{2}}F_0+\frac{d-3}{2}\gamma'\gamma^{-1}h\right|^2\,dx\,dt\\
&\leq 2\lambda^{d-3}\mathbb{E}\int_Q \theta^2\gamma^{d-3}F_0^2\,dx\,dt+C_2T^2\mathbb{E}\int_Q \theta^2\gamma^2 h^2\,dx\,dt.
\end{aligned}
\end{align*}
Then, it follows that
\begin{align}\label{inn2}
\begin{aligned}
&\,\mathbb{E}\int_Q \theta^2\left|(\lambda\gamma)^{\frac{d-3}{2}}F_0+\frac{d-3}{2}\gamma'\gamma^{-1}h\right|^2\,dx\,dt\\
&\leq 2\lambda^{d-3}\mathbb{E}\int_Q \theta^2\gamma^{d-3}F_0^2\,dx\,dt+C_2\lambda T^4\mathbb{E}\int_Q \theta^2\gamma^3 h^2\,dx\,dt.
\end{aligned}
\end{align}
Taking a large $\lambda\geq C_2T^2$ in \eqref{inn2}, we get that
\begin{align}\label{inn22}
\begin{aligned}
&\,\mathbb{E}\int_Q \theta^2\left|(\lambda\gamma)^{\frac{d-3}{2}}F_0+\frac{d-3}{2}\gamma'\gamma^{-1}h\right|^2\,dx\,dt\\
&\leq 2\lambda^{d-3}\mathbb{E}\int_Q \theta^2\gamma^{d-3}F_0^2\,dx\,dt+\frac{\lambda^3}{2C_1}\mathbb{E}\int_Q \theta^2\gamma^3 h^2\,dx\,dt.
\end{aligned}
\end{align}
Finally, by combining \eqref{inn1} and \eqref{inn22}, and taking a large enough $\lambda\geq C_2(T+T^2)$, we deduce the desired Carleman inequality \eqref{3.202002gene}.
\end{proof}
\subsection{Proof of Theorem \ref{thmm1.2}}
In this subsection, we adopt the following notation
$$\mathcal{L}_\mathcal{B}(d,z)=\lambda^d\mathbb{E}\int_0^T\int_\mathcal{B} \theta^2\gamma^d|z|^2 \,dx\,dt,$$
where $\mathcal{B}\subset G$ is any nonempty open subset. To prove Theorem \ref{thmm1.2}, we will first establish the following important result.
\begin{lm}\label{lmm2.3}
Under the assumptions of Theorem \ref{thmm1.2} and given $l\in\mathbb{N}$, $\varepsilon>0$, $k\in\{2,...,n\}$ and two any open sets $\mathcal{O}_0$ and $\mathcal{O}_1$ such that
$G_1\subset\mathcal{O}_1\Subset\mathcal{O}_0\subset \widetilde{G}_0$, there exists a positive constant $C_0$ (depending on $G,\mathcal{O}_0,\mathcal{O}_1, \mu_0,  a_0, \beta_0, \beta_{ij}$, and $M_0$) and $l_{kj}\in \mathbb{N}$, $1\leq j\leq k-1$ (depending only on $l,n,k$ and $j$), such that for all $z^0\in L^2_{\mathcal{F}_0}(\Omega;L^2(G;\mathbb{R}^n))$, the associated solution $z=(z_1,...,z_n)$ of system \eqref{1.1Adjoint} satisfies that 
\begin{align}\label{desiinee2.10}
\begin{aligned}
\mathcal{L}_{\mathcal{O}_1}(l,z_k)\leq&\, \varepsilon\left[\mathcal{I}(3(n+1-k),z_k)+\mathcal{I}(3(n-k),z_{k+1})\right]\\
&+C_0\left(1+\frac{1}{\varepsilon}\right)\sum_{j=1}^{k-1} \mathcal{L}_{\mathcal{O}_0}(l_{kj},z_j),
\end{aligned}
\end{align}
for all $$\lambda\geq\lambda_0=C_0\left(T+T^2+T^2\max_{i\leq j}\left(|a_{ij}|^{\frac{2}{3(j-i)+3}}_\infty+|C_{ij}|^{\frac{2}{3(j-i)+1}}_\infty+|b_{ij}|^{\frac{2}{3(j-i)+1}}_\infty\right)\right).$$
In \eqref{desiinee2.10}, we choose that $z_{n+1}=0$.
\end{lm}
\begin{proof}
Let us consider $\xi\in C^{\infty}(\mathbb{R}^N)$ such that
\begin{subequations}\label{assmzeta}
\begin{align}
&0\leq\xi\leq 1,\,\quad \xi=1 \,\, \textnormal{in}\,\, \mathcal{O}_1,\quad \,\textnormal{supp}(\xi)\subset \mathcal{O}_0,\label{assmzeta1}\\ &\quad
\frac{\Delta\xi}{\xi^{1/2}}\in L^\infty(G),\quad
\frac{\nabla\xi}{\xi^{1/2}}\in L^\infty(G;\mathbb{R}^N)\label{assmzeta2}.
\end{align}
\end{subequations}
Such a function $\xi$ exists. Indeed, by standard arguments, one can take $\xi_0 \in C^\infty_0(\mathbb{R}^N)$ satisfying \eqref{assmzeta1} and then choose $\xi = \xi_0^4$, which easily verifies \eqref{assmzeta}.

Let $k\in\{2,...,n\}$, from the assumption \eqref{asspmt} on the coefficients $a_{k,k-1}$, we assume without loss of generality that $a_{k,k-1}\geq a_0>0$ in $(0,T)\times\widetilde{G}_0$, $\mathbb{P}$\textnormal{-a.s}. In the rest of the proof, we fix $\delta_i>0$ ($i=0,1$), $m=2l-3(n-k)-1$ and $v=(\lambda\gamma)^l\theta^2$. It is easy to see that for a large $\lambda\geq C_0(T+T^2)$
\begin{align}\label{esttmforT}
|\nabla(v\xi)|\leq C_0\lambda^{l+1}\theta^2\gamma^{l+1}\xi^{1/2},\qquad |\partial_t v|\leq C_0\lambda^{l+2}\theta^2\gamma^{l+2}.
\end{align}
We first note that
\begin{align}\label{estfirl0}
a_0\mathcal{L}_{\mathcal{O}_1}(l,z_k)\leq\mathbb{E}\int_Q v\xi a_{k,k-1} |z_k|^2 \,dx\,dt.
\end{align}
By Itô's formula, we compute $d\langle z_{k-1},v\xi z_k\rangle_{L^2(G)}$, integrating the obtained equality on $(0,T)$ and taking the expectation on both sides, we obtain that
\begin{align}\label{desiinee2.102.11}
\begin{aligned}
\mathbb{E}\int_Q v\xi a_{k,k-1} |z_k|^2 \,dx\,dt=&\,\mathbb{E}\int_Q \partial_t v\,\xi z_{k-1}z_k \,dx\,dt+\mathbb{E}\int_0^T \langle L_k z_k,v\xi z_{k-1}\rangle \,dt\\
&+\sum_{j=1}^{k} \mathbb{E}\int_0^T \langle\nabla\cdot(C_{jk}z_j),v\xi z_{k-1}\rangle \,dt-\sum_{j=1}^{k} \mathbb{E}\int_Q v\xi a_{jk}z_j z_{k-1} \,dx\,dt\\
&-\mathbb{E}\int_Q v\xi a_{k+1,k}z_{k-1}z_{k+1}\,dx\,dt+\mathbb{E}\int_0^T \langle L_{k-1}z_{k-1},v\xi z_k\rangle \,dt\\
&-\sum_{j=1}^{k-1} \mathbb{E}\int_Q v\xi z_k a_{j,k-1}z_j \,dx\,dt+\sum_{j=1}^{k-1} \mathbb{E}\int_0^T \langle \nabla\cdot(C_{j,k-1}z_j),v\xi z_k\rangle \,dt\\
&+\sum_{j=1}^{k-1}\sum_{j'=1}^{k}\mathbb{E}\int_Q v\xi b_{j,k-1}b_{j'k}z_{j}z_{j'}\,dx\,dt\\
&=\sum_{i=1}^9 J_i,
\end{aligned}
\end{align}
where $\langle\cdot,\cdot\rangle$ denotes the duality product between $H^{-1}(G)$ and $H^1_0(G)$. Let us now estimate all the terms $J_i$ ($i=1,...,9$).

Using \eqref{esttmforT}, we have that
\begin{align*}
    \begin{aligned}
J_1&\leq \mathbb{E}\int_Q |\partial_t v| \xi |z_{k-1}||z_k| \,dx\,dt\\
&\leq C_0\lambda^{l+2}\mathbb{E}\int_Q  \theta^2\gamma^{l+2}\xi |z_{k-1}||z_k| \,dx\,dt\\
&\leq \frac{\delta_0}{7}\lambda^l\mathbb{E}\int_Q \theta^2\gamma^l\xi a_{k,k-1}|z_k|^2 \,dx\,dt+\frac{C_0}{\delta_0} \lambda^{l+4}\mathbb{E}\int_Q  \theta^2\gamma^{l+4}\xi|z_{k-1}|^2 \,dx\,dt.
    \end{aligned}
\end{align*}
Then, we get that
\begin{align}\label{ineeJ1}
    \begin{aligned}
J_1\leq \frac{\delta_0}{7}\mathbb{E}\int_Q v\xi a_{k,k-1}|z_k|^2 \,dx\,dt+\frac{C_0}{\delta_0} \lambda^{l+4}\mathbb{E}\int_Q  \theta^2\gamma^{l+4}\xi|z_{k-1}|^2 \,dx\,dt.
    \end{aligned}
\end{align}
By integration by parts, we obtain
\begin{align*}
    \begin{aligned}
J_2&=-\sum_{i,j=1}^N \mathbb{E}\int_Q \left(\xi z_{k-1}\beta_{ij}^k \frac{\partial z_k}{\partial x_j} \frac{\partial v}{\partial x_i}+ v z_{k-1}\beta_{ij}^k \frac{\partial z_k}{\partial x_j}\frac{\partial \xi}{\partial x_i}+v\xi\beta_{ij}^k \frac{\partial z_k}{\partial x_j}\frac{\partial z_{k-1}}{\partial x_i}\right) \,dx\,dt.
    \end{aligned}
\end{align*}
From \eqref{assmzeta}, it is easy to see that
\begin{align*}
    \begin{aligned}
J_2&\leq C_0\lambda^{l+1}\mathbb{E}\int_Q \theta^2 \gamma^{l+1}\xi|z_{k-1}||\nabla z_k| \,dx\,dt+C_0\lambda^l \mathbb{E}\int_Q\theta^2\gamma^l\xi^{1/2}|z_{k-1}||\nabla z_k| \,dx\,dt\\
&\quad+C_0\lambda^l \mathbb{E}\int_Q \theta^2\gamma^l\xi|\nabla z_k||\nabla z_{k-1}| \,dx\,dt.
    \end{aligned}
\end{align*}
By Young's inequality, it follows that for a large $\lambda\geq C_0T^2$
\begin{align}\label{ineeJ2}
    \begin{aligned}
J_2&\leq \frac{\varepsilon a_0}{6}\lambda^{3(n-k)+1}\mathbb{E}\int_Q \theta^2\gamma^{3(n-k)+1}|\nabla z_k|^2 \,dx\,dt+\frac{C_0}{\varepsilon}\Bigg(\lambda^{m+2}\mathbb{E}\int_Q \theta^2\gamma^{m+2}\xi|z_{k-1}|^2 \,dx\,dt\\
&\quad+\lambda^m\mathbb{E}\int_Q \theta^2\gamma^m\xi|\nabla z_{k-1}|^2 \,dx\,dt\Bigg).
    \end{aligned}
\end{align}
On the other hand, we have that
\begin{align}\label{j33ine}
    \begin{aligned}
        J_3&=\mathbb{E}\int_0^T \langle\nabla\cdot(C_{kk}z_k),v\xi z_{k-1}\rangle \,dt+\sum_{j=1}^{k-1} \mathbb{E}\int_0^T \langle\nabla\cdot(C_{jk}z_j),v\xi z_{k-1}\rangle \,dt\\
        &=-\mathbb{E}\int_Q \left[\nabla(v\xi)\cdot C_{kk}z_kz_{k-1} + v\xi C_{kk}\cdot\nabla z_{k-1} z_k\right] \,dx\,dt\\
        &\quad-\sum_{j=1}^{k-1} \mathbb{E}\int_Q \left[\nabla(v\xi)\cdot C_{jk}z_jz_{k-1}+v\xi C_{jk}\cdot\nabla z_{k-1} z_j\right]\,dx\,dt\\
        &=K_1+K_2.
    \end{aligned}
\end{align}
By using \eqref{esttmforT}, notice that
\begin{align*}
    \begin{aligned}
    K_1&\leq\mathbb{E}\int_Q |\nabla(v\xi)|  |C_{kk}| |z_k| |z_{k-1}| \,dx\,dt+\mathbb{E}\int_Q 
    |v\xi| |C_{kk}| |\nabla z_{k-1}| |z_k| \,dx\,dt\\
    &\leq C_0\lambda^{l+1}|C_{kk}|_\infty \mathbb{E}\int_Q \theta^2 \gamma^{l+1}\xi^{1/2}|z_k||z_{k-1}| \,dx\,dt+C_0\lambda^{l}|C_{kk}|_\infty \mathbb{E}\int_Q \theta^2 \gamma^{l}\xi |z_k||\nabla z_{k-1}| \,dx\,dt.
    \end{aligned}
\end{align*}
Then by Young's inequality, we have
\begin{align*}
    \begin{aligned}
K_1&\leq\frac{\delta_0}{7}\lambda^l\mathbb{E}\int_Q \theta^2\gamma^l\xi a_{k,k-1}|z_k|^2\,dx\,dt+\frac{C_0}{\delta_0}\Bigg(\lambda^{l+2}|C_{kk}|_\infty^2\mathbb{E}\int_0^T\int_{\mathcal{O}_0}\theta^2\gamma^{l+2}|z_{k-1}|^2\,dx\,dt\\
&\quad+\lambda^{l}|C_{kk}|_\infty^2\mathbb{E}\int_Q \theta^2\gamma^{l}\xi|\nabla z_{k-1}|^2 \,dx\,dt\Bigg).
    \end{aligned}
\end{align*}
Choosing a large $\lambda\geq C_0|C_{kk}|_\infty^2T^2$, we get that
\begin{align}\label{k122}
    \begin{aligned}
    K_1&\leq\frac{\delta_0}{7}\mathbb{E}\int_Q v\xi a_{k,k-1}|z_k|^2\,dx\,dt+\frac{C_0}{\delta_0}\Bigg(\lambda^{l+3}\mathbb{E}\int_0^T\int_{\mathcal{O}_0}\theta^2\gamma^{l+3}|z_{k-1}|^2\,dx\,dt\\
    &\quad+\lambda^{l+1}\mathbb{E}\int_Q \theta^2\gamma^{l+1}\xi|\nabla z_{k-1}|^2 \,dx\,dt\Bigg).
    \end{aligned}
\end{align}
We also have
\begin{align}\label{2.200ine}
    \begin{aligned}
    K_2&\leq\sum_{j=1}^{k-1}\mathbb{E}\int_Q\left[|\nabla(v\xi)||C_{jk}||z_j||z_{k-1}|+|v\xi||C_{jk}||\nabla z_{k-1}||z_j|\right]\,dx\,dt.
    \end{aligned}
\end{align}
Using \eqref{esttmforT} and applying Young's inequality in the right hand side of \eqref{2.200ine}, one has that
\begin{align*}
    \begin{aligned}
    K_2&\leq\sum_{j=1}^{k-1}\mathbb{E}\int_Q\left[C_0\lambda^{l+1}|C_{jk}|_\infty\theta^2\gamma^{l+1}\xi^{1/2}|z_j||z_{k-1}|+\lambda^{l}|C_{jk}|_\infty\theta^2\gamma^{l}\xi|z_j||\nabla z_{k-1}|\right]\,dx\,dt.
    \end{aligned}
\end{align*}
Then, it follows that
\begin{align}\label{k33}
    \begin{aligned}
    K_2&\leq C_0\Bigg(\lambda^m\mathbb{E}\int_Q \theta^2\gamma^m\xi|\nabla z_{k-1}|^2\,dx\,dt+\lambda^{m+2}\mathbb{E}\int_0^T\int_{\mathcal{O}_0} \theta^2\gamma^{m+2}|z_{k-1}|^2\,dx\,dt\\
    &\hspace{1cm}\,+\sum_{j=1}^{k-1}\lambda^{3(n-k)+1}|C_{jk}|^2_\infty\mathbb{E}\int_Q\theta^2\gamma^{3(n-k)+1}\xi|z_j|^2\,dx\,dt\Bigg).
    \end{aligned}
\end{align}
Taking a large $\lambda\geq C_0T^2|C_{jk}|_\infty^{2/(3(k-j)+1)}$ in \eqref{k33}, we get 
\begin{align}\label{kfi33}
    \begin{aligned}
    K_2&\leq C_0\bigg(\lambda^m\mathbb{E}\int_Q \theta^2\gamma^m\xi|\nabla z_{k-1}|^2\,dx\,dt+\lambda^{m+2}\mathbb{E}\int_0^T\int_{\mathcal{O}_0} \theta^2\gamma^{m+2}|z_{k-1}|^2\,dx\,dt\\
    &\hspace{1.2cm}+\sum_{j=1}^{k-1}\lambda^{3(n-j)+2}\mathbb{E}\int_Q\theta^2\gamma^{3(n-j)+2}\xi|z_j|^2\,dx\,dt\bigg).
    \end{aligned}
\end{align}
Combining \eqref{j33ine}, \eqref{k122} and \eqref{kfi33}, we conclude that
\begin{align}\label{ineeJ3}
    \begin{aligned}
J_3&\leq\frac{\delta_0}{7}\mathbb{E}\int_Q v\xi a_{k,k-1}|z_k|^2\,dx\,dt+\frac{C_0}{\delta_0}\bigg(\lambda^{l+3}\mathbb{E}\int_0^T\int_{\mathcal{O}_0}\theta^2\gamma^{l+3}|z_{k-1}|^2\,dx\,dt\\
&\quad+\lambda^{l+1}\mathbb{E}\int_Q \theta^2\gamma^{l+1}\xi|\nabla z_{k-1}|^2 \,dx\,dt\bigg)\\
&\quad+C_0\bigg(\lambda^m\mathbb{E}\int_Q \theta^2\gamma^m\xi|\nabla z_{k-1}|^2\,dx\,dt+\lambda^{m+2}\mathbb{E}\int_0^T\int_{\mathcal{O}_0} \theta^2\gamma^{m+2}|z_{k-1}|^2\,dx\,dt\\
&\hspace{1.4cm}+\sum_{j=1}^{k-1}\lambda^{3(n-j)+2}\mathbb{E}\int_Q\theta^2\gamma^{3(n-j)+2}\xi|z_j|^2\,dx\,dt\bigg).
    \end{aligned}
\end{align}
Similarly to $J_3$, see that
\begin{align*}
J_4=-\mathbb{E}\int_Q v\xi a_{kk} z_k z_{k-1} \,dx\,dt - \sum_{j=1}^{k-1} \mathbb{E}\int_Q v\xi a_{jk} z_j z_{k-1} \,dx\,dt,
\end{align*}
which implies that
\begin{align*}
    \begin{aligned}
        J_4&\leq \frac{\delta_0}{7}\lambda^l\mathbb{E}\int_Q \theta^2\gamma^l\xi a_{k,k-1}|z_k|^2\,dx\,dt+\frac{C_0}{\delta_0}\lambda^l|a_{kk}|^2_\infty\mathbb{E}\int_Q \theta^2\gamma^l\xi|z_{k-1}|^2\,dx\,dt\\
&\quad+C_0\lambda^{m+2}\mathbb{E}\int_Q\theta^2\gamma^{m+2}\xi|z_{k-1}|^2\,dx\,dt+C_0\sum_{j=1}^{k-1}\lambda^{3(n-k)-1}|a_{jk}|^2_\infty\mathbb{E}\int_Q\theta^2\gamma^{3(n-k)-1}\xi|z_j|^2\,dx\,dt.
    \end{aligned}
\end{align*}
By choosing a large $\lambda\geq C_0T^2|a_{jk}|_\infty^{2/(3(k-j)+3)}$, it follows that
\begin{align}\label{ineeJ4}
    \begin{aligned}
        J_4&\leq \frac{\delta_0}{7}\mathbb{E}\int_Q v\xi a_{k,k-1}|z_k|^2\,dx\,dt+\frac{C_0}{\delta_0}\lambda^{l+3}\mathbb{E}\int_Q \theta^2\gamma^{l+3}\xi|z_{k-1}|^2\,dx\,dt\\
&\quad+C_0\lambda^{m+2}\mathbb{E}\int_Q\theta^2\gamma^{m+2}\xi|z_{k-1}|^2\,dx\,dt+C_0\sum_{j=1}^{k-1}\lambda^{3(n-j)+2}\mathbb{E}\int_Q\theta^2\gamma^{3(n-j)+2}\xi|z_j|^2\,dx\,dt.
    \end{aligned}
\end{align}
By Young's inequality, we derive that
\begin{align}\label{ineeJ5}
    \begin{aligned}
        J_5\leq \frac{\varepsilon a_0}{2}\lambda^{3(n-k)}\mathbb{E}\int_Q\theta^2\gamma^{3(n-k)}|z_{k+1}|^2\,dx\,dt+\frac{C_0}{\varepsilon}\lambda^{m+1}\mathbb{E}\int_Q\theta^2\gamma^{m+1}\xi|z_{k-1}|^2\,dx\,dt.
    \end{aligned}
\end{align}
On the other hand, by integration by parts we have
\begin{align*}
    \begin{aligned}
        J_6&=-\sum_{i,j=1}^N\mathbb{E}\int_Q \beta_{ij}^{k-1}\frac{\partial z_{k-1}}{\partial x_j} \frac{\partial(v\xi z_k)}{\partial x_i}\,dx\,dt\\
        &=-\sum_{i,j=1}^N\mathbb{E}\int_Q \left(\beta_{ij}^{k-1}z_k\frac{\partial z_{k-1}}{\partial x_j}\frac{\partial(v\xi)}{\partial x_i}+\beta_{ij}^{k-1}v\xi\frac{\partial z_{k}}{\partial x_i}\frac{\partial z_{k-1}}{\partial x_j}\right)\,dx\,dt.
    \end{aligned}
\end{align*}
From \eqref{esttmforT}, it follows that
\begin{align*}
    \begin{aligned}
        J_6&\leq C_0\lambda^{l+1}\mathbb{E}\int_Q \theta^2\gamma^{l+1}\xi^{1/2}|z_k||\nabla z_{k-1}|\,dx\,dt+C_0\lambda^l\mathbb{E}\int_Q\theta^2\gamma^l\xi|\nabla z_k||\nabla z_{k-1}|\,dx\,dt,
    \end{aligned}
\end{align*}
which provides that 
\begin{align}\label{ineeJ6}
    \begin{aligned}
J_6&\leq\frac{\delta_0}{7}\mathbb{E}\int_Q v\xi a_{k,k-1}|z_k|^2\,dx\,dt+\frac{C_0}{\delta_0}\lambda^{l+2}\mathbb{E}\int_0^T\int_{\mathcal{O}_0} \theta^2\gamma^{l+2}|\nabla z_{k-1}|^2\,dx\,dt
\\
&\quad+\frac{\varepsilon a_0}{6}\lambda^{3(n-k)+1}\mathbb{E}\int_Q\theta^2\gamma^{3(n-k)+1}|\nabla z_k|^2\,dx\,dt+\frac{C_0}{\varepsilon}\lambda^{m}\mathbb{E}\int_Q\theta^2\gamma^{m}\xi|\nabla z_{k-1}|^2\,dx\,dt.
    \end{aligned}
\end{align}
Using Young's inequality, we have 
\begin{align}\label{j7inee}
    \begin{aligned}
J_7\leq\frac{\delta_0}{7}\lambda^l\mathbb{E}\int_Q\theta^2\gamma^{l}\xi a_{k,k-1}|z_k|^2\,dx\,dt+\frac{C_0}{\delta_0}\sum_{j=1}^{k-1}|a_{j,k-1}|^2_\infty\lambda^l\mathbb{E}\int_Q\theta^2\gamma^{l}\xi|z_j|^2\,dx\,dt.
    \end{aligned}
\end{align}
Taking a large $\lambda\geq C_0T^2|a_{j,k-1}|_\infty^{2/(3(k-j))}$ in \eqref{j7inee}, we find that
\begin{align}\label{ineeJ7}
    \begin{aligned}
J_7\leq\frac{\delta_0}{7}\mathbb{E}\int_Q v\xi a_{k,k-1}|z_k|^2\,dx\,dt+\frac{C_0}{\delta_0}\sum_{j=1}^{k-1}\lambda^{l+3(k-j)}\mathbb{E}\int_Q\theta^2\gamma^{l+3(k-j)}\xi|z_j|^2\,dx\,dt.
    \end{aligned}
\end{align}
By integration by parts, we obtain that
\begin{align*}
    \begin{aligned}
        J_8=-\sum_{j=1}^{k-1}\mathbb{E}\int_Q \nabla(v\xi)\cdot C_{j,k-1}z_j z_k \,dx\,dt-\sum_{j=1}^{k-1}\mathbb{E}\int_Q v\xi\nabla z_k\cdot C_{j,k-1}z_j \,dx\,dt.
    \end{aligned}
\end{align*}
From \eqref{esttmforT}, it follows that
\begin{align*}
    \begin{aligned}
J_8\leq&\, C_0\sum_{j=1}^{k-1}|C_{j,k-1}|_\infty\lambda^{l+1}\mathbb{E}\int_Q \theta^2\gamma^{l+1}\xi^{1/2} |z_j| |z_k| \,dx\,dt\\
&+\sum_{j=1}^{k-1}|C_{j,k-1}|_\infty\lambda^l\mathbb{E}\int_Q \theta^2\gamma^l\xi|\nabla z_k| |z_j| \,dx\,dt.
    \end{aligned}
\end{align*}
By Young's inequality, we have 
\begin{align}\label{inej8for2029}
    \begin{aligned}
        J_8&\leq \frac{\delta_0}{7}\lambda^l\mathbb{E}\int_Q \theta^2\gamma^l\xi a_{k,k-1}|z_k|^2\,dx\,dt+\frac{C_0}{\delta_0}\sum_{j=1}^{k-1}|C_{j,k-1}|_\infty^2\lambda^{l+2}\mathbb{E}\int_0^T\int_{\mathcal{O}_0} \theta^2\gamma^{l+2}|z_j|^2\,dx\,dt\\
&\quad+\frac{\varepsilon a_0}{6}\lambda^{3(n-k)+1}\mathbb{E}\int_Q \theta^2\gamma^{3(n-k)+1}|\nabla z_k|^2\,dx\,dt\\
&\quad+\frac{C_0}{\varepsilon}\sum_{j=1}^{k-1}\lambda^{m}|C_{j,k-1}|_\infty^2\mathbb{E}\int_Q\theta^2\gamma^{m}\xi|z_j|^2 \,dx\,dt.
    \end{aligned}
\end{align}
Taking a large $\lambda\geq C_0T^2|C_{j,k-1}|_\infty^{2/(3(k-j)-2)}$ in \eqref{inej8for2029}, we get that
\begin{align}\label{ineeJ8}
    \begin{aligned}
        J_8&\leq \frac{\delta_0}{7}\mathbb{E}\int_Q v\xi a_{k,k-1}|z_k|^2\,dx\,dt+\frac{C_0}{\delta_0}\sum_{j=1}^{k-1}\lambda^{l+3(k-j)}\mathbb{E}\int_0^T\int_{\mathcal{O}_0} \theta^2\gamma^{l+3(k-j)}|z_j|^2\,dx\,dt\\
&\quad+\frac{\varepsilon a_0}{6}\lambda^{3(n-k)+1}\mathbb{E}\int_Q \theta^2\gamma^{3(n-k)+1}|\nabla z_k|^2\,dx\,dt\\
        &\quad+\frac{C_0}{\varepsilon}\sum_{j=1}^{k-1}\lambda^{2l-3(n-2k+j+1)}\mathbb{E}\int_Q \theta^2\gamma^{2l-3(n-2k+j+1)}\xi|z_j|^2 \,dx\,dt.
    \end{aligned}
\end{align}
On the other hand, we also have 
\begin{align*}
    J_9=\sum_{j=1}^{k-1}\mathbb{E}\int_Q v\xi b_{j,k-1}b_{kk}z_jz_k \,dx\,dt+\sum_{j,j'=1}^{k-1}\mathbb{E}\int_Q v\xi b_{j,k-1}b_{j'k}z_jz_{j'} \,dx\,dt.
\end{align*}
Then, Young's inequality yields
\begin{align}\label{j9inee}
    \begin{aligned}
J_9&\leq\frac{\delta_0}{7}\lambda^l\mathbb{E}\int_Q\theta^2\gamma^l\xi a_{k,k-1}|z_k|^2 \,dx\,dt+\frac{C_0}{\delta_0}\sum_{j=1}^{k-1}\left(|b_{j,k-1}|_\infty^4+|b_{kk}|_\infty^4\right)\lambda^l\mathbb{E}\int_Q\theta^2\gamma^l\xi|z_{j}|^2 \,dx\,dt\\
&\quad+\sum_{j=1}^{k-1}|b_{j,k-1}|_\infty^2\lambda^l\mathbb{E}\int_Q\theta^2\gamma^l\xi|z_{j}|^2 \,dx\,dt+\sum_{j=1}^{k-1}|b_{jk}|_\infty^2\lambda^l\mathbb{E}\int_Q\theta^2\gamma^l\xi|z_{j}|^2 \,dx\,dt.
    \end{aligned}
\end{align}
By choosing a large $\lambda\geq C_0T^2|b_{jk}|_\infty^{2/(3(k-j)+1)}$ in \eqref{j9inee}, we find that
\begin{align}\label{ineeJ9}
    \begin{aligned}
J_9\leq&\,\frac{\delta_0}{7}\mathbb{E}\int_Q v\xi a_{k,k-1}|z_k|^2 \,dx\,dt+\frac{C_0}{\delta_0}\sum_{j=1}^{k-1}\Bigg(\lambda^{l+6(k-j)-4}\mathbb{E}\int_Q\theta^2\gamma^{l+6(k-j)-4}\xi|z_{j}|^2 \,dx\,dt\\
&+\lambda^{l+2}\mathbb{E}\int_Q\theta^2\gamma^{l+2}\xi|z_{j}|^2 \,dx\,dt\Bigg)+C_0\sum_{j=1}^{k-1}\lambda^{l+3(k-j)-2}\mathbb{E}\int_Q\theta^2\gamma^{l+3(k-j)-2}\xi|z_{j}|^2 \,dx\,dt\\
&+C_0\sum_{j=1}^{k-1}\lambda^{l+3(k-j)+1}\mathbb{E}\int_Q\theta^2\gamma^{l+3(k-j)+1}\xi|z_{j}|^2 \,dx\,dt.
    \end{aligned}
\end{align}
Combining \eqref{desiinee2.102.11}, \eqref{ineeJ1}, \eqref{ineeJ2}, \eqref{ineeJ3}, \eqref{ineeJ4}, \eqref{ineeJ5}, \eqref{ineeJ6}, \eqref{ineeJ7}, \eqref{ineeJ8} and \eqref{ineeJ9}, and choosing $\delta_0=1/2$, we deduce that
\begin{align}\label{322ine}
    \begin{aligned}
&\,\mathbb{E}\int_Q v\xi a_{k,k-1} |z_k|^2 \,dx\,dt\\
&\leq\varepsilon a_0\left[\mathcal{I}(3(n+1-k),z_k)+\mathcal{I}(3(n-k),z_{k+1})\right]\\
&\quad+C_0\left(1+\frac{1}{\varepsilon}\right)\Bigg[\lambda^\rho\mathbb{E}\int_0^T\int_{\mathcal{O}_0}\theta^2\gamma^{\rho}|z_{k-1}|^2\,dx\,dt\\
&\quad+\lambda^\sigma\mathbb{E}\int_Q\theta^2\gamma^{\sigma}\xi|\nabla z_{k-1}|^2\,dx\,dt+\sum_{j=1}^{k-2}\lambda^\nu\mathbb{E}\int_0^T\int_{\mathcal{O}_0}\theta^2\gamma^{\nu}|z_{j}|^2\,dx\,dt\Bigg],
    \end{aligned}
\end{align}
with
\begin{align*}
    &\rho=\max\left\{l+4,2l-3(n-k)+1,3(n-k)+5\right\},\\
    &\sigma=\max\left\{l+1,2l-1-3(n-k)\right\},\\
    &\nu=\max\left\{3(n-j)+2,2l-3(n-2k+j+1),l+6(k-j)-4,l+3(k-j)+1\right\}.
\end{align*}
Let us now absorb th term ``$\lambda^\sigma\mathbb{E}\displaystyle\int_Q\theta^2\gamma^{\sigma}\xi|\nabla z_{k-1}|^2\,dx\,dt$'' from the right hand side of \eqref{322ine}. In the rest of the proof, we also adopt the notation $\Tilde{v}=(\lambda\gamma)^\sigma\theta^2$. By Itô's formula, we first see that
\begin{align}\label{itofuplasproo}
d(\Tilde{v}\xi |z_{k-1}|^2)=\partial_t\Tilde{v}\,\xi |z_{k-1}|^2 \,dt+\Tilde{v}\xi\left(2z_{k-1} dz_{k-1}+\sum_{j,j'=1}^{k-1} b_{j,k-1}b_{j',k-1} z_j z_{j'}\,dt\right).
\end{align}
Integrating \eqref{itofuplasproo} on $Q$ and taking the expectation on both sides, we get that
\begin{align*}
    0=&\,\mathbb{E}\int_Q \partial_t\Tilde{v}\,\xi |z_{k-1}|^2 \,dx\,dt + \sum_{j,j'=1}^{k-1}\mathbb{E}\int_Q \Tilde{v}\xi  b_{j,k-1}b_{j',k-1} z_j z_{j'} \,dx\,dt\\
    &-2\mathbb{E}\int_Q \sum_{i,j=1}^N\beta_{ij}^{k-1}\frac{\partial z_{k-1}}{\partial x_i}\frac{\partial(\Tilde{v}\xi z_{k-1})}{\partial x_j} \,dx\,dt-2\mathbb{E}\int_Q \Tilde{v}\xi z_{k-1}\sum_{j=1}^{k-1} a_{j,k-1}z_j \,dx\,dt\\&-2\mathbb{E}\int_Q \sum_{j=1}^{k-1} z_j C_{j,k-1}\cdot\nabla(\Tilde{v}\xi z_{k-1})\,dx\,dt-2\mathbb{E}\int_Q \Tilde{v}\xi a_{k,k-1}z_{k-1}z_k \,dx\,dt,  
\end{align*}
which gives
\begin{align}\label{inneqII}
\begin{aligned}
    &\,2\mathbb{E}\int_Q \Tilde{v}\xi\sum_{i,j=1}^N\beta_{ij}^{k-1}\frac{\partial z_{k-1}}{\partial x_i}\frac{\partial z_{k-1}}{\partial x_j}\,dx\,dt\\
    &=\mathbb{E}\int_Q \partial_t\Tilde{v}\,\xi |z_{k-1}|^2 \,dx\,dt + \sum_{j,j'=1}^{k-1}\mathbb{E}\int_Q \Tilde{v}\xi  b_{j,k-1}b_{j',k-1} z_j z_{j'} \,dx\,dt\\
    &\quad-2\mathbb{E}\int_Q \sum_{i,j=1}^N\beta_{ij}^{k-1}\frac{\partial z_{k-1}}{\partial x_i}\frac{\partial (\Tilde{v}\xi)}{\partial x_j} z_{k-1}\,dx\,dt\\
    &\quad-2\mathbb{E}\int_Q \Tilde{v}\xi z_{k-1}\sum_{j=1}^{k-1} a_{j,k-1}z_j \,dx\,dt-2\mathbb{E}\int_Q \sum_{j=1}^{k-1} z_j C_{j,k-1}\cdot\nabla(\Tilde{v}\xi z_{k-1})\,dx\,dt\\
    &\quad-2\mathbb{E}\int_Q \Tilde{v}\xi a_{k,k-1}z_{k-1}z_k \,dx\,dt\\
    &=\sum_{i=1}^6 I_i.
    \end{aligned}
\end{align}
Let us now estimate the terms $I_i$ ($i=1,...,6$). From \eqref{esttmforT},  we first have that 
\begin{align}\label{ineeI1}
    I_1\leq C_0\lambda^{\sigma+2}\mathbb{E}\int_0^T\int_{\mathcal{O}_0} \theta^2\gamma^{\sigma+2} |z_{k-1}|^2 \,dx\,dt.
\end{align}
Taking a large $\lambda\geq C_0T^2|b_{j,k-1}|_\infty^2$
\begin{align}\label{ineeI2}
    I_2\leq C_0\lambda^{\sigma+1}\sum_{j=1}^{k-1}\mathbb{E}\int_0^T\int_{\mathcal{O}_0} \theta^2\gamma^{\sigma+1} |z_{j}|^2 \,dx\,dt.
\end{align}
By \eqref{esttmforT} and Young's inequality, we get
\begin{align}\label{ineeI3}
    I_3\leq \frac{\beta_0}{2}\mathbb{E}\int_Q \Tilde{v}\xi|\nabla z_{k-1}|^2\,dx\,dt+C_0\lambda^{\sigma+2}\mathbb{E}\int_0^T\int_{\mathcal{O}_0} \theta^2\gamma^{\sigma+2}|z_{k-1}|^2 \,dx\,dt.
\end{align}
Using again Young's inequality, we find
\begin{align*}
    I_4\leq C_0\lambda^{\sigma+3/2}\mathbb{E}\int_Q \theta^2\gamma^{\sigma+3/2}\xi|z_{k-1}|^2 \,dx\,dt+C_0\sum_{j=1}^{k-1}|a_{j,k-1}|_\infty^2\lambda^{\sigma-3/2}\mathbb{E}\int_Q \theta^2\gamma^{\sigma-3/2}\xi|z_{j}|^2 \,dx\,dt.
\end{align*}
Then for a large $\lambda\geq C_0T^2|a_{j,k-1}|_\infty^{2/(3(k-j))}$, we obtain that
\begin{align*}
I_4\leq&\, C_0\lambda^{\sigma+3/2}\mathbb{E}\int_Q \theta^2\gamma^{\sigma+3/2}\xi|z_{k-1}|^2 \,dx\,dt\\
&+C_0\sum_{j=1}^{k-1}\lambda^{\sigma-3/2+3(k-j)}\mathbb{E}\int_0^T\int_{\mathcal{O}_0} \theta^2 \gamma^{\sigma-3/2+3(k-j)}  |z_j|^2 \,dx\,dt,
\end{align*}
which implies that
\begin{align}\label{ineeI4}
I_4\leq C_0\sum_{j=1}^{k-1}\lambda^{\sigma-3/2+3(k-j)}\mathbb{E}\int_0^T\int_{\mathcal{O}_0} \theta^2 \gamma^{\sigma-3/2+3(k-j)}  |z_j|^2 \,dx\,dt.
\end{align}
On the other hand, we also have 
\begin{align*}
    I_5\leq&\, C_0\sum_{j=1}^{k-1}|C_{j,k-1}|_\infty\mathbb{E}\int_Q \Tilde{v}\xi|z_j||\nabla z_{k-1}| \,dx\,dt\\
    &+C_0\sum_{j=1}^{k-1}|C_{j,k-1}|_\infty\mathbb{E}\int_Q |\nabla(\Tilde{v}\xi)||z_j||z_{k-1}| \,dx\,dt.
\end{align*}
By \eqref{esttmforT} and Young's inequality, we obtain that
\begin{align}\label{I55239o}
\begin{aligned}
I_5\leq&\,\frac{\beta_0}{2}\mathbb{E}\int_Q \Tilde{v}\xi|\nabla z_{k-1}|^2 \,dx\,dt+C_0\sum_{j=1}^{k-1}|C_{j,k-1}|_\infty^2\lambda^\sigma\mathbb{E}\int_0^T\int_{\mathcal{O}_0} \theta^2\gamma^\sigma|z_j|^2 \,dx\,dt\\
&+C_0\lambda^{\sigma+3/2}\mathbb{E}\int_0^T\int_{\mathcal{O}_0} \theta^2\gamma^{\sigma+3/2}|z_{k-1}|^2 \,dx\,dt\\
&+C_0\sum_{j=1}^{k-1}\lambda^{\sigma+1/2}|C_{j,k-1}|_\infty^2\mathbb{E}\int_0^T\int_{\mathcal{O}_0}\theta^2\gamma^{\sigma+1/2}|z_j|^2 \,dx\,dt.
\end{aligned} 
\end{align}
Taking a large $\lambda\geq C_0T^2|C_{j,k-1}|_\infty^{2/{(3(k-j)-2)}}$ in \eqref{I55239o}, we deduce that
\begin{align}\label{ineeI5}
\begin{aligned}
    I_5\leq&\, \frac{\beta_0}{2}\mathbb{E}\int_{Q}\Tilde{v}\xi|\nabla z_{k-1}|^2 \,dx\,dt+C_0\sum_{j=1}^{k-1}\lambda^{\sigma+3(k-j)-2}\mathbb{E}\int_0^T\int_{\mathcal{O}_0}\theta^2\gamma^{\sigma+3(k-j)-2} |z_j|^2 \,dx\,dt\\&+C_0\sum_{j=1}^{k-1}\lambda^{\sigma-3/2+3(k-j)}\mathbb{E}\int_0^T\int_{\mathcal{O}_0}\theta^2\gamma^{\sigma-3/2+3(k-j)}|z_j|^2 \,dx\,dt.
    \end{aligned}
\end{align}
For the last term $I_6$, we first see that
\begin{align*}
    I_6\leq \delta_1\lambda^l\mathbb{E}\int_Q \theta^2\gamma^l\xi a_{k,k-1}|z_k|^2 \,dx\,dt+\frac{C_0}{\delta_1}\lambda^{2\sigma-l}\mathbb{E}\int_Q \theta^2\gamma^{2\sigma-l}\xi |z_{k-1}|^2 \,dx\,dt,
\end{align*}
which implies 
\begin{align}\label{ineeI6}
    I_6\leq \delta_1\mathbb{E}\int_Q v\xi a_{k,k-1}|z_k|^2 \,dx\,dt+\frac{C_0}{\delta_1}\lambda^{2\sigma-l}\mathbb{E}\int_0^T\int_{\mathcal{O}_0} \theta^2\gamma^{2\sigma-l} |z_{k-1}|^2 \,dx\,dt.
\end{align}
Recalling \eqref{assmponalpha} and combining \eqref{inneqII}, \eqref{ineeI1}, \eqref{ineeI2}, \eqref{ineeI3}, \eqref{ineeI4}, \eqref{ineeI5}, and \eqref{ineeI6}, we obtain for a large $\lambda\geq C_0T^2$
\begin{align}\label{lastII}
    \begin{aligned}
&\,\lambda^\sigma\mathbb{E}\int_Q\theta^2\gamma^{\sigma}\xi|\nabla z_{k-1}|^2\,dx\,dt\\&
\leq C_0\lambda^{\sigma+2}\mathbb{E}\int_0^T\int_{\mathcal{O}_0}\theta^2\gamma^{\sigma+2} |z_{k-1}|^2 \,dx\,dt+\frac{C_0}{\delta_1}\lambda^{2\sigma-l}\mathbb{E}\int_0^T\int_{\mathcal{O}_0}\theta^2\gamma^{2\sigma-l} |z_{k-1}|^2 \,dx\,dt\\
&\quad+\delta_1\mathbb{E}\int_Q v\xi a_{k,k-1}|z_k|^2 \,dx\,dt+C_0\sum_{j=1}^{k-1}\lambda^{\sigma-3/2+3(k-j)}\mathbb{E}\int_0^T\int_{\mathcal{O}_0}\theta^2\gamma^{\sigma-3/2+3(k-j)}|z_j|^2 \,dx\,dt.
    \end{aligned}
\end{align}
Therefore, by combining \eqref{322ine} and \eqref{lastII}, and choosing a small enough $\delta_1\leq C_0\varepsilon/(1+\varepsilon)$, we end up with
\begin{align*}
    \begin{aligned}
&\,\mathbb{E}\int_Q v\xi a_{k,k-1} |z_k|^2 \,dx\,dt\\
&\leq\varepsilon a_0\left[\mathcal{I}(3(n+1-k),z_k)+\mathcal{I}(3(n-k),z_{k+1})\right]\\
&\quad+C_0\left(1+\frac{1}{\varepsilon}\right)\Bigg[\lambda^\rho\mathbb{E}\int_0^T\int_{\mathcal{O}_0}\theta^2\gamma^{\rho}\xi|z_{k-1}|^2\,dx\,dt+\lambda^{\sigma+2}\mathbb{E}\int_0^T\int_{\mathcal{O}_0}\theta^2\gamma^{\sigma+2} |z_{k-1}|^2 \,dx\,dt\\
&\hspace{1cm}+\lambda^{2\sigma-l}\mathbb{E}\int_0^T\int_{\mathcal{O}_0}\theta^2\gamma^{2\sigma-l} |z_{k-1}|^2 \,dx\,dt+\sum_{j=1}^{k-2}\lambda^\nu\mathbb{E}\int_0^T\int_{\mathcal{O}_0}\theta^2\gamma^{\nu}\xi|z_{j}|^2\,dx\,dt\\
&\hspace{1cm}+\sum_{j=1}^{k-1}\lambda^{\sigma-3/2+3(k-j)}\mathbb{E}\int_0^T\int_{\mathcal{O}_0}\theta^2\gamma^{\sigma-3/2+3(k-j)}|z_j|^2 \,dx\,dt\Bigg],
    \end{aligned}
\end{align*}
which provides that
\begin{align}\label{ineefnl}
    \begin{aligned}
\mathbb{E}\int_Q v\xi a_{k,k-1} |z_k|^2 \,dx\,dt&\leq\varepsilon a_0\left[\mathcal{I}(3(n+1-k),z_k)+\mathcal{I}(3(n-k),z_{k+1})\right]\\
&\quad+C_0\left(1+\frac{1}{\varepsilon}\right)\sum_{j=1}^{k-1} \mathcal{L}_{\mathcal{O}_0}(l_{kj},z_j).
    \end{aligned}
\end{align}
Finally, by combining \eqref{ineefnl} and \eqref{estfirl0}, we conclude the desired estimate \eqref{desiinee2.10}. This completes the proof of Lemma \ref{lmm2.3}.
\end{proof}
We are now in position to prove Theorem \ref{thmm1.2}.
\begin{proof}[Proof of Theorem \ref{thmm1.2}]
Let us choose $G_1\Subset\widetilde{G}_0\subset G_0$ and $z=(z_1,...,z_n)$ be the solution of \eqref{1.1Adjoint}. By applying the Carleman estimate \eqref{3.202002gene} to each function $z_i$ ($1\leq i\leq n$) with $\mathcal{O}=G_1$, $L_0=L_i$, $d=3(n+1-i)$ and the terms
\begin{align*}
    &F_0=-\sum_{j=1}^i a_{ji}z_j-a_{i+1,i}z_{i+1},\quad F=\sum_{j=1}^i C_{ji}z_j,\quad F_1=-\sum_{j=1}^i b_{ji}z_j,\qquad1\leq i\leq n,
\end{align*}
with $z_{n+1}=0$. Then, we obtain for a large $\lambda\geq C_0(T+T^2)$.
\begin{align}\label{ineee2399}
\begin{aligned}
\mathcal{I}(3(n+1-i),z_i)\leq&\, C_0\bigg(\mathcal{L}_{G_1}(3(n+1-i),z_i)+\mathcal{I}(3(n-i),z_{i+1})\\
&\qquad+\sum_{j=1}^i \lambda^{3(n-i)}|a_{ji}|^2_\infty\mathbb{E}\int_Q \theta^2\gamma^{3(n-i)}|z_j|^2 \,dx\,dt\\
&\qquad+\sum_{j=1}^i\lambda^{3(n-i)+2}|b_{ji}|^2_\infty\mathbb{E}\int_Q \theta^2\gamma^{3(n-i)+2}|z_j|^2\,dx\,dt\\
&\qquad+\sum_{j=1}^i \lambda^{3(n-i)+2}|C_{ji}|^2_\infty\mathbb{E}\int_Q \theta^2\gamma^{3(n-i)+2}|z_j|^2 \,dx\,dt\bigg),\quad1\leq i\leq n.
\end{aligned}
\end{align}
By an iteration argument, we deduce from \eqref{ineee2399} that 
\begin{align}\label{ineee2399neww}
\begin{aligned}
&\,\sum_{i=1}^n\mathcal{I}(3(n+1-i),z_i)\\
&\leq C_0\Bigg(\sum_{i=1}^n\mathcal{L}_{G_1}(3(n+1-i),z_i)+\sum_{i=1}^n\sum_{j=1}^i \lambda^{3(n-i)}|a_{ji}|^2_\infty\mathbb{E}\int_Q \theta^2\gamma^{3(n-i)}|z_j|^2 \,dx\,dt\\
&\hspace{1.2cm}+\sum_{i=1}^n\sum_{j=1}^i\lambda^{3(n-i)+2}|b_{ji}|^2_\infty\mathbb{E}\int_Q \theta^2\gamma^{3(n-i)+2}|z_j|^2\,dx\,dt\\
&\hspace{1.2cm}+\sum_{i=1}^n\sum_{j=1}^i \lambda^{3(n-i)+2}|C_{ji}|^2_\infty\mathbb{E}\int_Q \theta^2\gamma^{3(n-i)+2}|z_j|^2 \,dx\,dt\Bigg).
\end{aligned}
\end{align}
Observe that one can absorb the last three terms in the right hand side of \eqref{ineee2399neww} using the left hand side. For that by recalling \eqref{2.301} and choosing a large $\lambda$ so that
$$\lambda\geq\lambda_0= C_0\left(T+T^2+T^2\max_{i\leq j}\left(|a_{ij}|_\infty^{\frac{2}{3(j-i)+3}}+|C_{ij}|_\infty^{\frac{2}{3(j-i)+1}}+|b_{ij}|_\infty^{\frac{2}{3(j-i)+1}}\right)\right),$$
we conclude that
\begin{align}\label{estim2441}
\begin{aligned}
\sum_{i=1}^n\mathcal{I}(3(n+1-i),z_i)\leq&\, C_0\sum_{i=1}^n\mathcal{L}_{G_1}(3(n+1-i),z_i).
\end{aligned}
\end{align}

Let us now use the estimate \eqref{desiinee2.10} to eliminate all the local integrals $\mathcal{L}_{G_1}(3(n+1-i),z_i)$ for $2\leq i\leq n$. To do so, let us choose open subsets $\widetilde{\mathcal{O}}_i\subset\widetilde{G}_0$ with $2\leq i\leq n$, so that
$$G_1\Subset\widetilde{\mathcal{O}}_n\Subset\widetilde{\mathcal{O}}_{n-1}\Subset...\Subset\Subset\widetilde{\mathcal{O}}_2\subset\widetilde{G}_0.$$
Applying the estimate \eqref{desiinee2.10} for $\mathcal{O}_1=G_1$, $\mathcal{O}_0=\widetilde{O}_n$, $k=n$, $l=3$, then for all $\varepsilon>0$, we obtain 
\begin{align}\label{inee1n}
\mathcal{L}_{G_1}(3,z_n)\leq \varepsilon \mathcal{I}(3,z_n)+C_0\left(1+\frac{1}{\varepsilon}\right)\sum_{j=1}^{n-1}\mathcal{L}_{\widetilde{\mathcal{O}}_n}(l_{nj},z_j).
\end{align}
Then by combining \eqref{estim2441} and \eqref{inee1n} and choosing a small $\varepsilon$, we deduce from \eqref{estim2441} that
\begin{align}\label{ineeznn1}
\begin{aligned}
\sum_{i=1}^n\mathcal{I}(3(n+1-i),z_i)\leq&\, C_0\sum_{i=1}^{n-1}\mathcal{L}_{\widetilde{\mathcal{O}}_n}(l_{ni},z_i).
\end{aligned}
\end{align}
We now apply the estimate \eqref{desiinee2.10} for $\mathcal{O}_1=\widetilde{\mathcal{O}}_n$, $\mathcal{O}_0=\widetilde{O}_{n-1}$, $k=n-1$, $l=l_{n,n-1}$, then for all $\varepsilon>0$, we get that
\begin{align}\label{inee2n}
\mathcal{L}_{\widetilde{\mathcal{O}}_n}(l_{n,n-1},z_{n-1})\leq \varepsilon [\mathcal{I}(6,z_{n-1})+\mathcal{I}(3,z_{n})]+C_0\left(1+\frac{1}{\varepsilon}\right)\sum_{j=1}^{n-2}\mathcal{L}_{\widetilde{\mathcal{O}}_{n-1}}(l_{n-1,j},z_j).
\end{align}
Combining \eqref{ineeznn1} and \eqref{inee2n} and taking small $\varepsilon$, we obtain 
\begin{align}\label{ineeznn1n1}
\begin{aligned}
\sum_{i=1}^n\mathcal{I}(3(n+1-i),z_i)\leq&\, C_0\sum_{i=1}^{n-2}\mathcal{L}_{\widetilde{\mathcal{O}}_{n-1}}(l_{ni},z_i).
\end{aligned}
\end{align}
Note that, from \eqref{estim2441} to \eqref{ineeznn1n1}, we have eliminated the terms $z_n$ and $z_{n-1}$ from the right hand side. Then by the same strategy of computations and in a finite steps, we deduce the desired Carleman estimate \eqref{carestimate}. This concludes the proof of Theorem \ref{thmm1.2}.
\end{proof}

\section{Null controllability result: Proof of Theorem \ref{thm01.1}}\label{sec3}

Let us first prove the observability inequality \eqref{observineqqke}.

\begin{proof}[Proof of Proposition \ref{propo1.1}]
From the Carleman estimate \eqref{carestimate}, we have that
\begin{align}\label{inn11}
\lambda^{3}\sum_{i=1}^n\mathbb{E}\int_{T/4}^{3T/4}\int_G \theta^2\gamma^{3}|z_i|^2\,dx\,dt\leq C_0 \lambda^l\mathbb{E}\int_{Q_0} \theta^2\gamma^l|z_1|^2 \,dx\,dt,
\end{align}
for any $$\lambda\geq\lambda_0=C_0\left(T+T^2+T^2\max_{i\leq j}\left(|a_{ij}|^{\frac{2}{3(j-i)+3}}_\infty+|C_{ij}|^{\frac{2}{3(j-i)+1}}_\infty+|b_{ij}|^{\frac{2}{3(j-i)+1}}_\infty\right)\right).$$
It is easy to see that there exists a constant $C_0>0$ such that for a large 
$\lambda\geq C_0T^2$
\begin{align}\label{ineqqthegam}
\begin{aligned}
&\,\lambda^3\theta^2\gamma^3\geq C_0 e^{-C_0\lambda T^{-2}}\qquad\textnormal{in}\quad (T/4,3T/4)\times G,\\
&\hspace{1.7cm}\lambda^l\theta^2\gamma^l\leq C_0\qquad\textnormal{in}\quad Q.
\end{aligned}
\end{align}
From \eqref{inn11} and \eqref{ineqqthegam}, we deduce that
\begin{align}\label{firstestim}
\sum_{i=1}^n\mathbb{E}\int_{T/4}^{3T/4}\int_G |z_i|^2\,dx\,dt\leq C_0e^{C_0\lambda T^{-2}} \mathbb{E}\int_{Q_0} |z_1|^2 \,dx\,dt,
\end{align}
for any
$$\lambda\geq\lambda_1=C_0\left(T+T^2+T^2\max_{i\leq j}\left(|a_{ij}|_\infty^{\frac{2}{3(j-i)+3}}+|C_{ij}|_\infty^{\frac{2}{3(j-i)+1}}+|b_{ij}|_\infty^{\frac{2}{3(j-i)+1}}\right)\right).$$
In the rest of the proof, we fix $\lambda=\lambda_1$. On the other hand, we need to derive an energy estimate for solutions of system \eqref{1.1Adjoint}. Let $t\in(0,T)$, by Itô's formula, we compute 
$d|z_i|_{L^2(G)}^2$ (with $1\leq i\leq n$), then we integrate the obtained equality on $(t,T)$ and taking expectation on both sides, we obtain
\begin{align*}
\begin{aligned}
&\,\mathbb{E}|z_i(T,\cdot)|_{L^2(G)}^2-\mathbb{E}|z_i(t,\cdot)|_{L^2(G)}^2\\
&=- 2\mathbb{E}\int_t^T\int_G \sum_{i',j'=1}^N \beta_{i'j'}^i\frac{\partial z_i}{\partial x_{i'}}\frac{\partial z_i}{\partial x_{j'}}\,dx\,dt-2\mathbb{E}\int_t^T\int_G \sum_{j=1}^i a_{ji}z_j z_i\,dx\,dt\\
&\hspace{0.5cm}-2\mathbb{E}\int_t^T\int_G \sum_{j=1}^i z_j C_{ji}\cdot\nabla z_i\,dx\,dt-2\mathbb{E}\int_t^T\int_G a_{i+1,i}z_iz_{i+1}\,dx\,dt\\
&\hspace{0.5cm}+\sum_{j,j'=1}^i\mathbb{E}\int_t^T\int_G b_{ji}b_{j'i}z_j z_{j'}\,dx\,dt,\qquad 1\leq i\leq n,
\end{aligned}
\end{align*}
with $z_{n+1}=0$. Then, it follows that for any $\varepsilon>0$
\begin{align}\label{ineeq1}
\begin{aligned}
&\,\sum_{i=1}^n\mathbb{E}|z_i(T,\cdot)|_{L^2(G)}^2-\sum_{i=1}^n\mathbb{E}|z_i(t,\cdot)|_{L^2(G)}^2\\
&\leq - 2\beta_0\sum_{i=1}^n\mathbb{E}\int_t^T\int_G |\nabla z_i|^2\,dx\,dt+\sum_{i=1}^n\sum_{j=1}^i|a_{ji}|_\infty\mathbb{E}\int_t^T\int_G |z_j|^2 \,dx\,dt
\\
&\quad+\sum_{i=1}^n\sum_{j=1}^i|a_{ji}|_\infty\mathbb{E}\int_t^T\int_G |z_i|^2 \,dx\,dt+\varepsilon n\sum_{i=1}^n\mathbb{E}\int_t^T\int_G |\nabla z_i|^2 \,dx\,dt\\
&\quad+\frac{1}{\varepsilon}\sum_{i=1}^n\sum_{j=1}^i|C_{ji}|^2_\infty\mathbb{E}\int_t^T\int_G |z_j|^2 \,dx\,dt
+C_0\sum_{i=1}^n\mathbb{E}\int_t^T\int_G |z_i|^2\,dx\,dt\\
&\quad+\sum_{i=1}^n\sum_{j=1}^i|b_{ji}|_\infty^2\mathbb{E}\int_t^T\int_G |z_j|^2\,dx\,dt.
\end{aligned}
\end{align}
Choosing a small $\varepsilon>0$ in \eqref{ineeq1}, it is easy to see that 
\begin{align*}
\begin{aligned}
&\,\sum_{i=1}^n\mathbb{E}|z_i(T,\cdot)|_{L^2(G)}^2-\sum_{i=1}^n\mathbb{E}|z_i(t,\cdot)|_{L^2(G)}^2\\
&\leq C_0\left(1+\max_{i\leq j}\left(|a_{ij}|_\infty+|C_{ij}|^2_\infty+|b_{ij}|^2_\infty\right)\right)\sum_{i=1}^n\mathbb{E}\int_t^T\int_G |z_i|^2\,dx\,dt.
\end{aligned}
\end{align*}
Therefore by  Gronwall's inequality, it follows that
\begin{align}\label{lastbyGrn}
\mathbb{E}\int_G |z(T,x)|^2 dx\leq e^{C_0T(1+\max_{i\leq j}(|a_{ij}|_\infty+|C_{ij}|_\infty^2+|b_{ij}|^2_\infty))}\sum_{i=1}^n\mathbb{E}\int_G|z_i(t,x)|^2 dx.
\end{align}
Integrating \eqref{lastbyGrn} on $(T/4,3T/4)$ and combining the obtained inequality with \eqref{firstestim}, we conclude the desired  observability inequality \eqref{observineqqke}.
\end{proof}
We are now ready to establish our null controllability result of system \eqref{1.1}.
\begin{proof}[Proof of Theorem \ref{thm01.1}]
Fix $\varepsilon>0$ and $y^T\in L^2_{\mathcal{F}_T}(\Omega;L^2(G;\mathbb{R}^n))$. Let us introduce the following optimal control problem
\begin{align}\label{contrPbb}
\inf\{J_\varepsilon(u),\;u\in L^2_\mathcal{F}(0,T;L^2(G))\},
\end{align}
with
$$J_\varepsilon(u)=\frac{1}{2}\mathbb{E}\int_Q u^2 \,dx\,dt+\frac{1}{2\varepsilon}\mathbb{E}\int_G |y(0)|^2dx,$$
where $(y,Y)$ is the solution of \eqref{1.1} associated with the control $u$ and the final state $y^T$. It is easy to see that the problem \eqref{contrPbb} admits a unique optimal solution $u_\varepsilon$ such that
\begin{align}\label{contruu}
u_\varepsilon=\chi_{G_0}(x)z_\varepsilon^1, 
\end{align}
where $z_\varepsilon=(z_\varepsilon^1,...,z_\varepsilon^n)$ is the solution of the equation
\begin{equation}\label{soluuzeps}
\begin{cases}
\begin{array}{ll}
dz_\varepsilon-L(t)z_\varepsilon \,dt=\left(-A^*z_\varepsilon+\nabla\cdot(C^*z_\varepsilon)\right)\,dt-B^*z_\varepsilon \,dW(t)&\textnormal{in}\,\,Q,\\
z_\varepsilon=0&\textnormal{on}\,\,\Sigma,\\
z_\varepsilon(0)=\frac{1}{\varepsilon}y_\varepsilon(0)&\textnormal{in}\,\,G,
			\end{array}
		\end{cases}
\end{equation}
with $(y_\varepsilon,Y_\varepsilon)$ is the solution of \eqref{1.1} associated to the control $u_\varepsilon$ and the final state $y^T$.

On the other hand, using Itô's formula for $d(y_\varepsilon,z_\varepsilon)_{L^2(G;\mathbb{R}^n)}$, we obtain that
\begin{align*}
    \mathbb{E}\int_Q \chi_{G_0}(x)u_\varepsilon z_\varepsilon^1 \,dx\,dt+\frac{1}{\varepsilon}\mathbb{E}\int_G |y_\varepsilon(0,x)|^2 dx=\mathbb{E}\int_G (y^T,z_\varepsilon(T))_{\mathbb{R}^n} dx.
\end{align*}
Recalling \eqref{contruu}, it follows that
\begin{align*}
    \mathbb{E}\int_{Q_0} \left|z_\varepsilon^1\right|^2 \,dx\,dt+\frac{1}{\varepsilon}\mathbb{E}\int_G |y_\varepsilon(0,x)|^2 dx\leq\frac{e^{C_0K}}{2}\mathbb{E}\int_G \left|y^T\right|^2 dx+\frac{1}{2e^{C_0K}}\mathbb{E}\int_G |z_\varepsilon(T,x))|^2 dx,
\end{align*}
where $e^{C_0K}$ is the observability constant in \eqref{observineqqke}. Using again the inequality \eqref{observineqqke} for solutions of \eqref{soluuzeps}, we deduce that
\begin{align}\label{uniffestm}
    \mathbb{E}\int_Q |u_\varepsilon|^2 \,dx\,dt+\frac{2}{\varepsilon}\mathbb{E}\int_G |y_\varepsilon(0,x)|^2 dx\leq e^{C_0K}\mathbb{E}\int_G \left|y^T\right|^2 dx.
\end{align}
Then there exists a sub-sequence $\widetilde{u}_\varepsilon$ of $u_\varepsilon$ such that
$$\widetilde{u}_\varepsilon\longrightarrow\widehat{u}\quad\textnormal{weakly in}\quad L^2_\mathcal{F}(0,T;L^2(G)),\;\;\textnormal{as}\;\;\varepsilon\rightarrow0.$$
Notice that $\textnormal{supp}\, \widehat{u}\subset [0,T]\times \overline{G_0}$. Let us now denote by $(\widehat{y},\widehat{Y})$ the solution of system \eqref{1.1} associated to the control $\widehat{u}$ and the terminal condition $y^T$. Then from \eqref{uniffestm}, we conclude that $\widehat{y}(0,\cdot)=0$ in $G$, $\mathbb{P}$-a.s., and that the control $\widehat{u}$ satisfies the desired inequality \eqref{1.2201}. This achieves the proof of Theorem \ref{thm01.1}.
\end{proof}

\section{Conclusion}
In this paper, we have studied the null controllability of the system \eqref{1.1} under the assumptions \eqref{assmAC} and \eqref{asspmt}. This system is a prototype of cascade systems of coupled backward stochastic parabolic equations, incorporating both reaction and convection terms with bounded coefficients, as well as general second-order parabolic operators under standard conditions. To achieve this, we first prove a Carleman estimate for the associated adjoint coupled forward stochastic parabolic system \eqref{1.1Adjoint}, using a general Carleman estimate for forward stochastic parabolic equations with a parameter \( d \in \mathbb{R} \) and a drift term in the negative Sobolev space \( H^{-1} \)-space. As a result, we derive the required observability inequality and, through a classical duality argument, conclude our null controllability result. Furthermore, we provide an explicit estimate for the null control cost with respect to the final time \( T \) and the potentials \( a_{ij}, b_{ij}, \) and \( C_{ij} \) (with \( 1 \leq i \leq j \leq n \)).

\end{document}